\documentclass[12pt,twoside]{amsart}

\usepackage{amsmath}
\usepackage[english]{babel}
\usepackage{amsfonts}
\usepackage{amssymb}
\usepackage{enumerate}
\usepackage{geometry}
\usepackage{graphicx}
\usepackage{color}
\usepackage[mathscr]{euscript}
\usepackage{amsthm}
\usepackage{array}

\newcommand{\cm}{\mathrm{cm}}

\newcommand{\A}{\mathbb{A}}
\newcommand{\B}{\mathbb{B}}
\newcommand{\C}{\mathbb{C}}
\newcommand{\dia}{\diamondsuit}

\newcommand{\uhr}{\!\!\upharpoonright\!\!}

\DeclareMathOperator{\NF}{succ}

\DeclareMathOperator{\RO}{RO}

\DeclareMathOperator{\orb}{orb}

\DeclareMathOperator{\cf}{cf}
\DeclareMathOperator{\CF}{CF}

\DeclareMathOperator{\Aut}{Aut}

\DeclareMathOperator{\id}{id}
\DeclareMathOperator{\hgt}{ht}

\newtheorem{thm}{Theorem}[section]
\newtheorem{lm}[thm]{Lemma}
\newtheorem{prp}[thm]{Proposition}

\theoremstyle{definition}
\newtheorem{defi}[thm]{Definition}
\newtheorem{expl}[thm]{Example}

\theoremstyle{remark}
\newtheorem{rem}[thm]{Remark}

\begin{document}
\title{Souslin Algebra Embeddings}
\author{Gido Scharfenberger-Fabian}
\address{Ernst-Moritz-Arndt-Universit\"at\\
Robert-Blum-Stra\ss e 3\\
17489 Greifswald\\
Germany\\
Phone/Fax: 0049 - 3834 - 86 4638/-15}
\email{gido.scharfenberger-fabian@uni-greifswald.de}
\date{\today}

\begin{abstract}
A Souslin algebra is a complete Boolean algebra
whose main features are ruled by a tight combination
of an antichain condition with an infinite distributive law.

The present article divides into two parts.
In the first part a representation theory
for the complete and atomless subalgebras
of Souslin algebras is established
(building on ideas of Jech and Jensen).
With this we obtain some basic results
on the possible types of subalgebras
and their interrelation.

The second part begins with a review of some generalizations
of results from descriptive set theory concerning Baire category
which are then used
in non-trivial Souslin tree constructions that yield Souslin algebras
with a remarkable subalgebra structure.
In particular, we use this method to prove
that under the diamond principle there is a
bi-embeddable though not isomorphic
pair of homogeneous Souslin algebras.
\end{abstract}
\maketitle

\section*{Introduction}

Souslin trees are well-known to most set theorists,
Souslin algebras not so well any more,
and subalgebras of Souslin algebras in general
are suspected of being a messy business.
I would like to put in a good word for them here.

We consider $\kappa$-Souslin algebras
(for definitions see the following section)
as well as their representations
by normal $\kappa$-Souslin trees and address the problem how to
describe and classify complete embeddings between $\kappa$-Souslin algebras.
After introducing the representation for Souslin subalgebras
we give a rough classification of the possible types of embeddings and
find implications between existence statements involving them
as well as counter examples proving non-implications.

The representation theory of Part 1 of this text
was primarily developed exclusively for the case
$\kappa=\aleph_1$ in my PhD thesis
(\cite{diss}\footnote{Readers familiar with \cite{diss} should note
1.) that the notion of tree equivalence relation (t.e.r.) as defined in this 
article corresponds to what is called a \emph{decent} t.e.r. in \cite{diss},
and 
2.) that only Part 1 and Sections \ref{sec:reduction}
and \ref{sec:rigid-non-rigid}
of Part 2 consist of generalizations of results given in \cite{diss}
while the remainder of Part 2 brings new material.})
in order to establish the
consistency of the existence of a chain homogeneous $\aleph_1$-Souslin algebra
(cf. the subsequent paper \cite{mcsa}).
Here we take $\kappa$ to be any regular cardinal.
To describe Souslin subalgebras we
define the notion of  a \emph{tree equivalence relation}
on a $\kappa$-normal tree.
The properties considered for the classification of embeddings are mainly 
\emph{niceness} (introduced by Jensen, cf. \cite{devlin-johnsbraten})
and \emph{largeness}
and the global negations thereof.
For example, we show that large subalgebras are always nice
(Theorem \ref{thm:large}) and
that the existence of a nice and nowhere large subalgebra always implies
that there also is a nowhere nice subalgebra
(Theorem \ref{thm:exist_non-nice}).
We also study connections between
the symmetric structure of $\kappa$-Souslin algebras
(i.e. its automorphisms)
and its subalgebra structure.
Apart from the applications as given in this text,
the representation theory might also be useful
in the study of intermediate models of generic extensions
built using $\kappa$-Souslin trees or $\kappa$-Souslin algebras
or in other related areas of set theoretic research.

The topological notions developed
at the beginning of Part 2
in Section \ref{sec:dst} serve to facilitate
the choice of the relevant limit levels
in the $\kappa$-Souslin tree constructions
and provide a nice tool for involved diagonalization procedures.
The argument is a refinement of the
\emph{Diagonal Principle} as formulated
by Jech in \cite[p.63]{jech_automorphisms}:
\begin{quotation}
If $T$ is a countable normal tree of limit length, then there exists a branch
through $T$ which satisfies a countable number of prescribed conditions.\\
(E.g., given a countable set $B$ of branches of $T$, there exists a branch
which is not in $B$.)
\end{quotation}
Here we observe that the set of relevant branches is (a generalization of)
a Polish space and ``conditions'' are comeagre subsets.
The idea to use Baire category for Souslin tree constructions
is not at all new.
Taking into account the correspondence between
Baire category and Cohen reals it was implicitely used by
Jensen in his countable models constructions of Souslin trees
(cf. \cite[Chapters IV and V]{devlin-johnsbraten})
whose generic branches are close to being Cohen reals.
Or, for a more recent example, it applies along with
the parametrized diamond principles for Baire category
as considered in \cite{hrusak-dzamonja-moore}.

Nevertheless, this tool has, as far as I know,
not yet been used to perform advanced Souslin tree constructions.
We use it to construct a rigid $\mu^+$-Souslin algebra
with non-rigid Souslin subalgebras (Section \ref{sec:rigid-non-rigid}),
a rigid $\mu^+$-Souslin algebra with an essentially  unique Souslin subalgebra
(Section \ref{sec:unique_non-nice}) and 
a pair of $\mu^+$-Souslin algebras that forms a counter example to the
Schr\"oder-Bernstein-Theorem for $\mu^+$-Souslin algebras
(Section \ref{sec:no_schroeder}).
In a subsequent paper (\cite{mcsa}) I will present constructions of  
chain homogeneous $\aleph_1$-Souslin algebras
in which the same method is applied.
Of course, non of these constructions can be carried out in ZFC alone.
We do not intend to give an exhaustive picture of what can be done under
varying hypothesis and simply
assume variants of the well-known diamond principle $\dia$.

The paper is fairly self-contained, though of course some acquaintance with
Souslin tree constructions is an advantage for the reader.

\section{Preliminaries}

Our notation and terminology follow mainly
\cite{koppelberg} and \cite{devlin-johnsbraten} (Boolean)
and \cite{jechneu} (set theoretic,
exception: we use $\varphi"M$
to denote the image of the set $M$ under the mapping $\varphi$).

All Boolean algebras considered in this text are complete
and all subalgebras are tacitly assumed to be \emph{regular},
i.e., if $A$ is a subalgebra of $B$ and
for some $M\subset A$ the infimum $\sum^A M$
with respect to $A$ exists
(and it does as $A$ is assumed to be complete),
then it coincides with the sum $\sum^B M$ taken in $B$.
If $M$ is a subset of the Boolean algebra $B$, then $\langle M\rangle^\cm$
denotes \emph{the subalgebra of $B$ completely generated by $M$},
i.e., the least complete subalgebra of $B$ which contains $M$ as a subset.

A frequently used item is the \emph{canonical (upper) projection}
of a (complete) Boolean $B$ algebra onto its 
subalgebra $A$:
$$h = h_{B,A}: B \to A,\quad b \mapsto \sum\{ a\in A\mid\, ab=a\}.$$
As usual we omit subscripts if there is no danger of confusion.
Note that $h$ is not a homomorphism as it only respects sums but neither
products nor complements in general.

Whenever we talk of the \emph{natural ordering} of a Boolean algebra $B$,
we mean the relation defined by $a \leq_B b :\iff ab = a$.
We denote the \emph{relative algebra of $B$ with respect to $b$} by
$B\uhr b = \{ a\in B \mid\, a\leq b\}$.
When a (complete) subalgebra $A$ of $B$ and an element $b\in B$ are given,
we might also consider the algebra of products
$bA = b\cdot A = \{ba\,\mid\, a\in A\}$ which is a (complete)
subalgebra of $B\uhr b$.
In this situation the projection $h=h_{B,A}$ gives rise
to an isomorphism between $bA$ and $A\uhr h(b)$,
the inverse map being multiplication with $b$.

A (complete) Boolean algebra $B$ is called
\begin{itemize}
\item[]\emph{simple}, if it has no atomless (complete) subalgebras,
\item[]\emph{rigid}, if it admits no automorphisms except for the
identical map,
\item[]\emph{homogeneous}, if for every $b\in B^+$
the relative algebra $B\uhr b$ is isomorphic to $B$.
\end{itemize}

\subsection{$\kappa$-Souslin algebras}

Let $\kappa$ be an regular uncountable cardinal.
An \emph{antichain} of a Boolean algebra is a subset consisting of pairwise
disjoint elements,
and the \emph{$\kappa$-chain condition} states that
every antichain is of cardinality less than $\kappa$.
A \emph{$\kappa$-Souslin algebra} is a complete Boolean algebra that
satisfies both the $\kappa$-chain condition and the
\emph{$(\kappa,\infty)$-distributive law}, i.e.,
for index sets $I,J$ where $|I|<\kappa$ and $J$ is arbitrary and
each family $(a_{ij}\mid\, i\in I,\,j\in J )$ of elements
the following equation holds:
$$\sum_{i\in I}\prod_{j\in J}a_{ij}=
\prod\left\{\sum_{i\in I} a_{if(i)}\mid f\in {{^I}J}\right\}.$$
This distributive law also has valuable characterizations in terms of
common refinements of \emph{paritions of unity}, i.e. maximal antichains
(cf. \cite[Propositions 14.8/9]{koppelberg}):
A Boolean algebra is $(\kappa,\infty)$-distributive if and only if
every family $(X_i)_{i\in I}$ of less than $\kappa$ maximal antichains
has a common refinement, i.e., there is a maximal antichain $X$ such that
for every $i\in I$ and each member $a\in X_i$ there is
an element $b\in X$ with $ab=b$, i.e.,
$b$ lies below $a$ in the natural partial ordering of $B$.
As a consequence, if $B$ is $(\kappa,\infty)$-distributive and 
$A$ is a subalgebra of $B$ which is completely generated by fewer than
$\kappa$ elements, then $A$ is atomic.
These two results are heavily used when a $\kappa$-Souslin algebra
is represented as the regular open algebra of a $\kappa$-Souslin tree.

Note that every atomless (and complete) subalgebra of a $\kappa$-Souslin
algebra is $\kappa$-Souslin itself. We therefore call these subalgebras
\emph{Souslin subalgebras} (omitting the parameter $\kappa$ as it is
determined by the context).

A result concerning $\kappa$-Souslin algebras and well-known only in the case
where $\kappa=\aleph_1$ is Solovay's barrier for the cardinality
of $\kappa$-Souslin algebras (cf. \cite[Theorem 30.20]{jechneu}):
A $\kappa$-Souslin algebra can have at most $2^\kappa$ elements.
We will not use this result here
as we concentrate on $\kappa$-Souslin algebras that can be represented by
$\kappa$-Souslin trees and therefore always are of cardinality $2^{<\kappa}$.

\subsection{Trees}

A \emph{tree} is a partial order $(T,<_T)$ with the additional property,
that for every element $t\in T$,
the set of its predecessors, $\{s<_T t\}:=\{s\in T\mid s<_Tt\}$,
is well-ordered by the ordering $<_T$.
Whenever possible, we omit the subscript $_T$
and denote the tree ordering just by $<$.

The elements of a tree are called its \emph{nodes}, the minimal
elements are \emph{roots}.
The \emph{height} of a node, $\hgt t$, is the order type of the well-order
$(\{s<t\},<)$.
Nodes of limit height are also called \emph{limit nodes}.
If $\hgt t = \gamma +1$ is a successor ordinal, then we denote by
$t^- :=  t\uhr\gamma$ the immediate predecessor of $t$.

For every node $t$ we define the set of its immediate successors,
$$\NF t:=\{s\in T\mid t<s\text{ and }\hgt s=\hgt t+1\}.$$
For a cardinal $\mu$ we say that $T$ is \emph{$\mu$-splitting}
if every node has exactly $\mu$ immediate successors.
For every ordinal $\alpha$ we define the $\alpha$th level of $T$ and denote it by
$T_\alpha:=\{t\in T\mid \hgt t=\alpha\}$.
The \emph{height of} $T$ is the minimal ordinal $\alpha$ such that $T_\alpha$ is empty.
For a subset $C$ of $\hgt T$ we consider the tree
$$T\uhr C=\bigcup_{\alpha\in C}T_\alpha$$
with the ordering $<$ inherited from $T$ and call this tree
\emph{the restriction of $T$ to (the levels from) $C$}.
If $t\in T_\alpha$ and $\gamma<\alpha$ then $t\uhr\gamma$
denotes the unique predecessor of $t$ on level $\gamma$.

A subset $b$ of a tree $T$ is a \emph{branch}
if it is closed downwards and linearly ordered by $<$.
The \emph{length} $\ell(b)$ of a branch $b$
is just its order type with respect to $<$.
We sometimes take branches to be maps $\ell(b)\to b$ enumerating the nodes
in a monotone way.
A branch $x\subset T$ of limit length $\lambda$ is \emph{extended} if there
is a node $t\in T_\lambda$ that dominates all members of $x$:
$t>s$ for all $s\in x$.
A branch is \emph{cofinal} if its length coincides with $\hgt T$.
An \emph{antichain} of $T$ is a subset
that consists of pairwise incomparable nodes.
We call branches or antichains \emph{maximal} if they cannot be extended.
Note, that every (non-empty) level of $T$ is a maximal antichain.

A tree $T$ is $\mu$-closed if all branches in $T$, whose length has cofinality
less than $\mu$, are extended.
In particular, a $\mu$-closed tree has no maximal branches with cofinality
less than $\mu$.

A tree $T$ is \emph{normal} if the following hold:
\begin{itemize}
\item $T$ has a unique root,
\item every node $t$ has at least two successors on every level
$T_\alpha$ with $\hgt t<\alpha<\hgt T$
\item branches of limit length $\lambda$ have at most one extension to level
$T\lambda$ (the \emph{unique limits} condition)
\end{itemize}
A tree $T$ is \emph{$\mu$-normal} if it is normal and
every level of $T$ has less than $\mu$ nodes.

For every node $t\in T$ we let $T(t):=\{s\in T\mid t\leq s\}$
and call it the \emph{tree $T$ relativized to $t$}.
A \emph{homogeneous} tree is a tree $T$, that admits tree isomorphisms
between $T(s)$ and $T(t)$ for all pairs $s,t$ of nodes from the same
level $T_\alpha$ of $T$.
A \emph{rigid} tree has no tree automorphism but the
identical map.
Operations on trees sometimes used in the text are the \emph{tree product}
and the \emph{tree sum}
$$S \otimes T \,:=\, \bigcup_{\gamma<\alpha}S_\gamma\times T_\gamma
\qquad\text{ and }\qquad
 S \oplus T\, :=\, \{\mathbf{root}\}\cup S \, \dot{\cup}\, T\, , $$
where $\alpha=\hgt T = \hgt S$ and {\bf root} is a new node,
equipped with the obvious orderings.

The apparatus used in Part 2 of the article rests entirely
on the following definition, which is albeit useful also in Part 1:
For a normal tree $T$ of limit height
let $[T]$ be the set of cofinal branches of $T$.
We topologise $[T]$ with the basis that consists of the sets 
$\hat{s} := \{ x\in [T]\,\mid\, s \in x\}$ for all $s\in T$.
With this topology $[T]$ is a regular Hausdorff (i.e. T$_3$) space
of weight $|T|$.
Moreover, if $T$ is an $\aleph_1$-normal tree of countable limit height,
then $[T]$ is a Polish space, i.e.,
it is completely metrizable and second countable.

\subsection{$\kappa$-Souslin trees}

Now let $\kappa$ be an uncountable, regular cardinal.
A \emph{$\kappa$-Souslin tree} is a tree of height $\kappa$
that has neither antichains nor branches of size $\kappa$.
Note, that a $\kappa$-normal tree of height $\kappa$
is $\kappa$-Souslin if and only if it has no cofinal branches.
A \emph{subtree} is a subset which is a union of branches,
i.e., it is closed downwards.
(For example, $\{s<t\}\cup T(t)$ is always a subtree of $T$.)
Every $\kappa$-Souslin tree has a normal subtree which is $\kappa$-Souslin.
In this text we only consider normal Souslin trees.

The following Subtree Lemma is well-known for the case $\kappa=\omega_1$ but
its proof (as given, e.g., in \cite{larson})
literally translates to the general, regular case.
It captures the content of the notion of a $\kappa$-Souslin tree without
recourse to related structures such as Souslin lines or $\kappa$-Souslin
algebras.
\begin{lm}[Subtree Lemma]
\label{lm:subtree}
Let $\kappa$ be an uncountable, regular cardinal and $T$ a normal
$\kappa$-Souslin tree.
If $S$ is a subtree of $T$ with $|S|=\kappa$ then $S$ contains a subtree
$\{s<t\}\cup T(t)$ for some $t\in T$.
\end{lm}

In order to turn a tree into a Boolean algebra we provide it with the
(reversed) partial order topology:
The basic open sets are $T(s)$ for $s\in T$.
Then we simply take the regular open algebra $\RO T$
of the space $T$ with this topology.
The basic representation lemma for Souslin algebras is given by (the proof of)
\cite[Theorem 14.20]{koppelberg}:
\begin{enumerate}[1)]
\item For every (normal) $\kappa$-Souslin tree $T$
its regular open algebra $\RO T$ is $\kappa$-Souslin, and
\item for every $\kappa$-Souslin algebra $\B$,
if $\B$ is completely generated by $\kappa$ many of its elements,
then there is a (normal) $\kappa$-Souslin tree $T$ which can be
(with reversed order) regularly embedded onto a dense subset of $\B$.
\end{enumerate}
We stress once more that in the present paper
only $\kappa$-Souslin algebras are considered
which are completely generated by trees as in (2) above.
Following \cite{devlin-johnsbraten}
we call a subset $T$ of the $\kappa$-Souslin algebra $\B$
a \emph{Souslinization of $\B$} if $(T,>_\B)$ is a normal $\kappa$-Souslin
tree and the limit nodes in $T$ are obtained as products over their
predecessors: $s = \prod\{ t\in T\mid\, t >_\B s\}$.
A minor inconvenience of this terminology is that we regard trees as growing
upwards while Souslinizations grow downwards with respect to
the natural Boolean order $\leq_\B$ of $\B$.
If possible we prefer the tree order view, i.e.,
the common phrase ``$t$ is above $s$''
is tantamount to ``$s$ is closer to the root than $t$''
or in Boolean notation to $t\leq_B s$.

Two Souslinizations of the same $\kappa$-Souslin algebra can look quite
different, e.g. 2-splitting vs. infinitarily splitting.
However, by the following Restriction Lemma they always coincide on a
club set of levels. We will use this fact in a considerable portion of proofs.
\begin{lm}[Restriction Lemma]
\label{lm:restriction}
If the $\kappa$-Souslin algebras $\A$ and $\B$ are souslinized by
$S$ and $T$ respectively and if $\varphi:\A \to \B$ is an isomorphism,
then there is a club set $C\subset\kappa$ such that
the restriction of $\varphi$ to $S\uhr C$ is an isomorphism onto $T\uhr C$.
\end{lm}
For a proof take the one of \cite[Lemma 25.6]{jech}
(or a solution to \cite[Exercise 30.15]{jechneu})
and translate ``countable'' to ``less than $\kappa$''.

\subsection{$\dia$-principles}

For a cardinal $\kappa$ and a stationary subset $E\subset\kappa$
we denote the following statement by $\dia_\kappa(E)$:
\begin{quotation}
There is a sequence $(R_\alpha)_{\alpha<E}$ (the \emph{$\dia$-sequence})
such that for every subset $X$ of $\kappa$ the set
$$\{\alpha \in E \mid X \cap \alpha = R_\alpha \}$$
is stationary in $\kappa$.
\end{quotation}
The principle $\dia_\kappa(E)$ implies that $2^{\kappa}=\kappa^+$ and is
therefore not a theorem of ZFC.
But for many stationary sets $E$ it follows from
G\"odel's axiom of constructibility and can be made true by forcing.
We will use this principle in situations where $\kappa$ is a successor cardinal
$\kappa = \mu^+$ with $\mu = \mu^{<\mu}$
and $E = \CF_\mu = \{\alpha < \kappa \,\mid\, \cf(\alpha)=\mu\}$.

\part{Elementary representation and classification
of Souslin subalgebras}

Throughout all of Part 1 let
$\kappa$ denote a regular uncountable cardinal.

\section{Tree equivalence relations}
\label{sec:ter}

Subalgebras of Souslin algebras have been considered before,
e.g.\ in \cite{jech_simple} or \cite[\S 5]{bekkali-bonnet},
\cite{koppelberg-monk}
and more implicitly in \cite{devlin-johnsbraten} or \cite[\S 8]{larson}.
To represent a subalgebra $\A$ of the Souslin algebra $\B$
with respect to some Souslinization $T$ of $\B$,
the first three sources define
a \emph{good equivalence relation} on the Souslinization $T$,
while the last two use maps between trees $T\uhr C$
(for some club set $C\subseteq\omega_1$) and a Souslinization $S$ of $\A$.

We combine the two approaches in so far
as we will consider equivalence relations,
which are designed in a way such that they directly induce
the relevant mappings between the Souslinizations.

\begin{defi}
\begin{enumerate}[a)]
\item Let $T$ be a $\kappa$-normal tree of height $\mu\leq\kappa$.
An equivalence relation $\equiv$ on $T$
is a \emph{tree equivalence relation (t.e.r.)} if
\begin{enumerate}[i)]
\item
$\equiv$ respects levels, i.e., $s\equiv t$ only if $\hgt_T s=\hgt_T t$;
\item
$\equiv$ is compatible with $<_T$, i.e., for $s<_T s'$ and $t<_T t'$
with $s$ and $t$ of the same height, 
$s'\equiv t'$ implies $s\equiv t$;
\item
the induced partial order on the set $T/\!\equiv$ of $\equiv$-cosets given by
$$a<_{T/\equiv}b\iff (\exists s\in a,\,t\in b)s<_T t$$
for $a,b\in T/\!\equiv$ is a $\kappa$-normal tree order;
\item $\equiv$ is \emph{honest}, by which we mean that for all triples $(s,s',t)$ of nodes
$s\equiv t$ in some level $\gamma$ of $T$ and $s'>_T s$ the following holds:
If there is no successor of $t$  that is equivalent to $s'$, then the
same holds already for $s'\uhr(\gamma+1)$, i.e.,
there is no $t'\in T_{\gamma+1}$ above $t$ equivalent to $s'\uhr(\gamma+1)$.

\end{enumerate}
\item If $T$ souslinizes $\B$ and $\A$ is a Souslin
subalgebra of $\B$,
we say that the t.e.r. $\equiv$ on $T$
\emph{represents $\A$ on $T$}
if the sums over the $\equiv$-classes form a dense subset of $\A$:
$$\left\langle \sum s/\!\equiv\,\mid s\in T\right\rangle^\cm=\A.$$
\end{enumerate}
\end{defi}

\begin{rem}
\begin{enumerate}
\item Note that in point (iii) the tree $T/\!\!\equiv$ has unique limits.
This implies that on a limit level $T_\alpha$ the t.e.r. $\equiv$ is
completely determined by its behavior on $T\uhr\alpha$ below.
\item Furthermore, as the tree order on $T/\!\!\equiv$ splits in every node,
we get that every t.e.r. represents an \emph{atomless},
i.e. a Souslin subalgebra.
\item Call a triple $(s,s',t)$ of nodes a \emph{dispute (on $\equiv$)}
if $s\equiv t$ and $s<s'$ yet there is no successor of $t$ equivalent to $s'$,
i.e., $(s,s',t)$ is as in the definition of honesty above.
Then $\equiv$ is honest if and only if for every dispute $(s,s',t)$ on $\equiv$
already $(s,s'\uhr(\hgt(s)+1),t)$ is a dispute.
This is illustrated in figure \ref{fig:honest+nice}.
\item Honesty prevents a t.e.r.
from associating two nodes of level $T_\gamma$
that can be distinguished
by the subalgebra that the t.e.r. represents.
In partucular, if a $\kappa$-Souslin tree carries two different t.e.r.s,
then the subalgebras represented by these t.e.r.s differ as well.
\end{enumerate}
\end{rem}

\begin{figure}[h]
\begin{center}
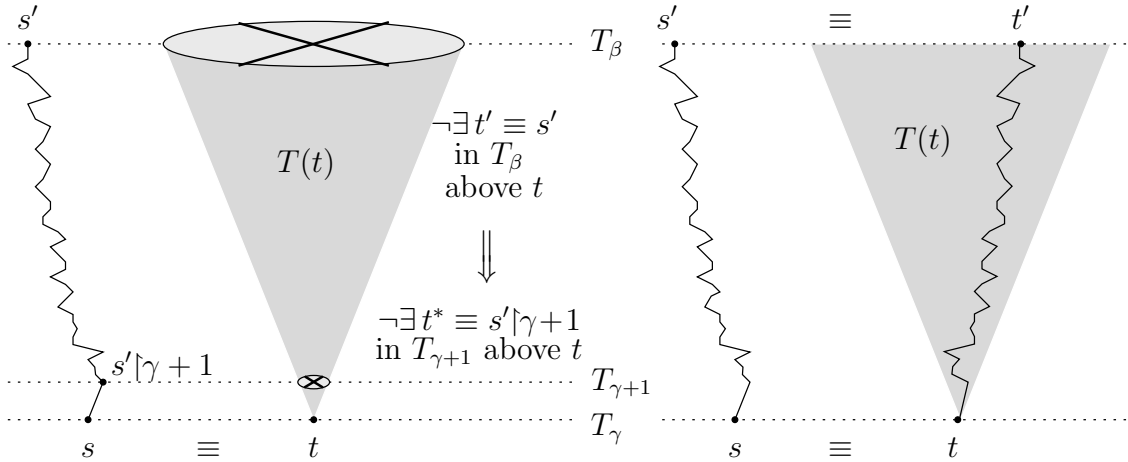
\caption{Honesty of a tree equivalence relation ---
the dispute case on the left hand side versus the nice case on the right}
\label{fig:honest+nice}
\end{center}
\end{figure}

For the moment, let us denote by 
\emph{pre-t.e.r.} an equivalence relation on a tree which
satisfies conditions (i-iii) above but not necessarily honesty.

Part b) of the following proposition
gives us a necessary criterion
for testing whether a pre-t.e.r. is honest
with respect to a limit level $T_\alpha$.
With its aid we can destroy unwanted
t.e.r.s/sub\-algebras in recursive Souslin algebra constructions
during which we have to choose appropriate limit levels of a tree
(cf.~Example \ref{expl:simple} and Theorem \ref{thm:unique_non-nice}).

\begin{prp}
\label{prp:irred-dense}
Let $T$ be a $\kappa$-normal tree of height $\beta\le\kappa$ carrying a
t.e.r. $\equiv$.
Let $\alpha<\beta$ be a limit ordinal.
Consider the equivalence relation
$\simeq$ on $[T\uhr\alpha]$ induced by $\equiv$ through
$$x\simeq y\quad:\iff\quad(\forall \gamma<\alpha)x\uhr\gamma=y\uhr\gamma.$$
\begin{enumerate}[a)]
\item The $\simeq$-classes are closed subsets of $[T]$.
\item For $s\in T_\alpha$ denote by $x_s$ the branch $\{r\in T\uhr\alpha\mid
r<s\}\in [T\uhr\alpha]$.
For each branch $x\in [T]$ consider its class $x/\!\!\simeq$ as a
subspace of $[T]$.
Then for every $s\in T_\alpha$ the $\alpha$-branches associated
to the members of the $\equiv$-class of $s$,
i.e. the set
$$\{x_r\mid r\in T_\alpha \text{ and } r\equiv s\},$$
lies densely in the corresponding class $x_s/\!\!\simeq$.
Stated in more elementary terms,
for every node $s\in T_\alpha$,
branch $y\simeq x_s$ and ordinal $\gamma<\alpha$ there is a node
$t$ in level $T_\alpha$ such that $t\equiv s$ and $t > y\uhr\gamma$.
\end{enumerate}
\end{prp}
\begin{proof}
Part a) follows easily from the fact
that for each $x\in[T]$ the set
$$S^x = \bigcup x/\!\!\simeq\,\,=
\{s\in T \mid\,(\exists\, y\simeq x)\,s\in y\}$$
is a subtree of $T$ and $x/\!\!\simeq\,\, = [S^x]$, which is always closed.

To prove b) by contradiction,
assume that for $s,y$ and $\gamma$ as above there is no $t>y\uhr\gamma$,
$t\equiv s$.
Then the triple $(s\uhr\gamma, s, y\uhr\gamma)$ would constitute a
dispute on $\equiv$, but $s\uhr(\gamma+1) \equiv y\uhr (\gamma+1)$.
This contradicts point (iv) of the last definition.
\end{proof}

\begin{rem}
  \begin{enumerate}
  \item Note that, while in Proposition \ref{prp:irred-dense} we used different
symbols for the t.e.r. $\equiv$ and the induced equivalence relation $\simeq$
on the space of branches of length $\alpha$
(because here this difference was crucial)
we will further on denote the induced relation
with the same symbol as the t.e.r.
(in most cases: $\equiv$).
  \item In some of the later arguments we will identify the branches
of the form $x_s$ with the corresponding nodes $s$.
  \item Given an equivalence relation $\equiv$
on some topological space $\mathscr{X}$,
call a subset $M \subset \mathscr{X}$ \emph{suitable for $\equiv$}
if for every member $x\in M$
the intersection $M \cap (x/\!\!\equiv)$
is a dense subset of the space $x/\!\!\equiv$.
With this notion at hand,
the conclusion of Proposition \ref{prp:irred-dense}.b) reads as:
\begin{quotation}
The set of branches $\{x_s\mid s\in T_\alpha\}$ corresponding to the nodes
of level $T_\alpha$ is suitable for the equivalence relation induced by
$\equiv$ on the $\alpha$-branches of $T$.
\end{quotation}
  \end{enumerate}
\end{rem}

Jensen defined a subalgebra $\A$ of a $\kappa$-Souslin algebra $\B$ to be a
\emph{nice subalgebra}
if there is some Souslinization $T$ of $\B$ such that the image of $T$ under
the canonical projection $h:\B\to\A,\, b\mapsto \prod\{a\in\A\mid b\le a\}$
is a Souslinization of $\A$.
We now define the corresponding notion for t.e.r.s.
\begin{defi}
\begin{enumerate}[a)]
\item A t.e.r. $\equiv$ on $T$ is called \emph{nice},
if for all $s,s',t$ in $T$ with
$s<_T s'$ and $s\equiv t$ there is some $t'>_T t$ with $s'\equiv t'$.
\item A t.e.r. $\equiv$ on $T$ is called \emph{almost nice},
if for all $s,s',t$ in $T$ with
$s<_T s'$ and $s\equiv t$ and $\hgt(s)=\alpha+1$
for some $\alpha$ there is some $t'>_T t$ with $s'\equiv t'$.
\end{enumerate}
\end{defi}

\begin{rem}
\begin{enumerate}
\item Obviously niceness is the complete absence of disputes
and almost niceness means that no dipute may have its lower nodes
in a successor level of the tree.
So both properties imply the honesty of the t.e.r.
(from now on we can forget about pre-t.e.r.s).
\item Honesty and niceness are handed down to any restriction to
a club set of levels while almost niceness is not,
because such a restriction can turn a limit level into a successor level.
On the other hand it is easy to see,
that every t.e.r. can be obtained as a restriction of
an almost nice t.e.r. to some club set of levels.
\item 
It is easy to see that the nice subalgebras (with respect to Jensen's
definition) are exactly those that can be
represented by nice t.e.r.s.
Given a t.e.r. $\equiv$ on a Souslinization $T$ of $\B$
let us denote the associated projection by
$$\pi_\equiv:T\to\B,\quad t\mapsto\sum t/\!\!\equiv.$$
The t.e.r. $\equiv$ is nice if and only if $\pi_\equiv = h \uhr T$,
and it is almost nice if and only if $\pi_\equiv$ and $h$ coincide
on all successor levels of $T$.
\end{enumerate}
\end{rem}

The next lemma will be called \emph{the Representation Lemma for
Souslin subalgebras}.
\begin{lm}
\label{lm:representation}
Let $\A$ be a Souslin subalgebra of the $\kappa$-Souslin algebra $\B$,
and let $S$ be any Souslinization of $\B$.
\begin{enumerate}[a)]
\item There is a Souslinization $T$ of $\B$ that admits an almost nice t.e.r.
  $\equiv$ representing $\A$.
\item There are a club $C\subseteq\kappa$ and a t.e.r.
$\equiv$ on $S\uhr C$ such that $\equiv$ represents $\A$.
\item If $\A$ is furthermore nice and represented by $\equiv$ on $S$
  then there is a club $C\subseteq\kappa$ such that $\equiv$
  is nice on $S\uhr C$.
\end{enumerate}
\end{lm}

\begin{proof}
We only prove part a) since
parts b) and c) follow directly from part a) by the Restriction Lemma.
Before constructing $T$ and $\equiv$ by recursion,
we describe a method of refining
a given partition $P$ of unity in $\B$
to a partition $R$ in $\B$ with the property,
that $h''R$ is a partition in $\A$.
Let $Q$ be the set of atoms of $\langle h''P\rangle^\cm \subset \A$
and define
$$R=\{pq\mid p\in P,\,q\in Q\}\setminus\{0\}.$$
Then $R$ refines $P$, and
for $pq\in R$ we have $h(pq)=qh(p)=q$ since $q$ is an atom.
So $h''R=Q$.

Now fix a dense subset $\{x_{\alpha+1}\mid\alpha\in\kappa\}$ of $\B$
indexed by successor ordinals.
Starting with the root level $T_0=\{1_\B\}$
let $P_\alpha$ be any partition in $\B$ refining $T_\alpha$ in
such a way that every $s\in T_\alpha$ 
is divided in at least two parts,
for all $s\in T_\alpha$ the image $h(s)$
is not equal to the $h$-images of the parts of $s$,
and $x_\alpha\in\langle P_\alpha\rangle^\cm$.
Then let $T_{\alpha+1}$ be the refinement
of $P_\alpha$ with respect to $h$ as described above.
So $h''T_{\alpha+1}$ is a partition in $\A$.
The limit levels of $T$ are canonically defined as
$$T_\alpha:=\left\{\prod b\mid b\in [S\uhr\alpha]\right\}\setminus\{0\}.$$
Thus $T$ is a Souslinization of $\B$.
The t.e.r. is then given on successor levels by
$$s\equiv t:\!\iff h(s)=h(t).$$
This also determines $\equiv$ on the limit levels
and defines an almost nice t.e.r. on $T$.
\end{proof}

As an illustration of the notion of t.e.r. and a first application
of the Representation Lemma we reformulate
Jech's construction of a simple $\kappa$-Souslin algebra,
i.e., one having no non-trivial Souslin subalgebra.
(cf.~\cite{jech_simple}). 

\begin{expl}[a simple $\kappa$-Souslin algebra]
\label{expl:simple}
We construct a Souslinization $T$ of a simple $\kappa$-Souslin algebra $\B$.
We assume that $\kappa=\mu^+$ is a
successor cardinal and $\mu^{<\mu}=\mu$ and $\dia_\kappa(\CF_\mu)$
hold\footnote{A similar construction
(which also applies to an inaccessible cardinal $\kappa$
that is not weakly compact) under an appropriate
($\square + \dia$)-assumption is of course possible but more cumbersome,
cf.~\cite[Theorem VII.1.3]{devlin} for that framework.}.
Let $(R_\nu)_{\cf(\nu)=\mu}$ be a $\dia$-sequence.

We will define a $\kappa$-normal and $\mu$-closed $\kappa$-Souslin tree
order on the set $\kappa$.
We let $0$ be the root and provide every node of $T$ with $\mu$ direct
successors such that level $T_\alpha$ consists of the ordinal interval
between\footnote{To be correct, $T_1=\mu\setminus\{0\}$,
$T_n= \mu\cdot n \setminus\mu(n-1)$ for $n\in\omega\setminus\{0,1\}$
and $T_\alpha = \mu(\alpha+1)\setminus \mu\alpha$
for all $\alpha \in \kappa \setminus \omega$.}
$\mu\cdot\alpha$ and $\mu\cdot(\alpha+1)=\mu\cdot\alpha +\mu$.

We take full limits on limit levels $\alpha$ of cofinality $<\mu$,
i.e. we extend all branches of length $\alpha$.
Thanks to our our hypothesis on cardinal arithmetics
there are only $\mu$ branches to extend,
so our tree remains $\kappa$-normal.

On limit stage $\alpha$ of cofinality $\mu$ we consider the space $[T\uhr\alpha]$
of cofinal branches through $T\uhr\alpha$
and have to choose a dense subset of cardinality $\mu$
subject to some further restrictions imposed by our
$\dia_\kappa(\CF_\mu)$-sequence $(R_\nu)_{\cf\nu=\mu}$.
If $\alpha<\mu\alpha$ then we can extend $T\uhr\alpha$ by choosing any dense
subset $Q$ of $[T\uhr\alpha]$ of size $\mu$ and extending the branches in $Q$
to $T_{\alpha+1}$.

In the case where $\alpha=\mu\alpha$
we ask the $\dia$-sequence for some information about $T\uhr\alpha$.
We let the first bit of $R_\alpha$ decide whether we care about antichains or
about t.e.r.s.
If $0\in R_\alpha$ and $A=R_\alpha\setminus\{0\}$ is a maximal antichain of
$T\uhr\alpha$ then we choose our dense subset $Q$ from the dense open set
$P=\{x\in[T\uhr\alpha]\,\mid\, (\exists t\in A) t\in x\}$ to guarantee
that $A$ is still a maximal antichain when considered as a subset of
$T_{\alpha+1}$. 

If $0\notin R_\alpha$ and if $R_\alpha$ codes a t.e.r. $\equiv$ on $T\uhr C$ 
for some club set $C\subset\alpha$ then we
want to choose the new level $T_\alpha$ in a way that destroys $\equiv$, i.e., 
the unique extension of $\equiv$ to $T_{\alpha+1}$ violates the honesty
criterion of Proposition \ref{prp:irred-dense}. 
For this consider the equivalence
relation induced by $\equiv$ on the space $[T\uhr\alpha]$ of cofinal
branches via
$$x\equiv y\quad :\iff \quad (\forall \gamma\in C)x\uhr\gamma\equiv
y\uhr\gamma.$$
The $\equiv$-classes of $[T\uhr\alpha]$ 
are closed and nowhere dense subsets of $[T\uhr\alpha]$.
If $\equiv$ is a non-trivial t.e.r. then there is certainly a $\equiv$-class
$x/\!\!\equiv$ of size $>1$.
Fix a representative $x$ of such a class.
In order to define $T_\alpha$ we choose
a dense subset $Q$ of $[T\uhr\alpha]\setminus (x/\!\!\equiv)$ of elements
not equivalent to $x$.
Then extend every branch in $Q\cup\{x\}$.
The node extending $x$
violates the conclusion of Proposition \ref{prp:irred-dense}, so the extension
of $\equiv$ to $T\uhr C\cup\{\alpha\}$ is no longer honest and therefore
no t.e.r.

If $\cf(\alpha)=\mu$ yet $R_\alpha$ neither is an antichain nor does it code
some t.e.r. on $T$, then we simply choose any dense $\mu$-subset $Q$ of
$[T\uhr\alpha]$ and extend the branches in $Q$ to level $T_\alpha$.
This finishes the recursive construction of $T$.
 
By standard $\dia_\kappa$-arguments,
the result $T$ of this construction is a $\kappa$-Souslin
tree that admits no t.e.r.\  
So by the Representation Lemma $\B=\RO T$ has no
proper and atomless complete subalgebra.
\end{expl}

Note that, while in the above construction
we explicitely talk about a \emph{non-trivial} t.e.r.,
we will from now on tacitly assume the t.e.r.s proposed by a $\dia$-sequence
not to be trivial, i.e., not to be the identity.

We close this section with a proposition on the local nature of niceness.
For this and also for later purposes,
we say that a Souslin subalgebra $\A$ is \emph{nowhere nice} in the
$\kappa$-Souslin algebra $\B$ if for every $b\in \B^+$
the relative subalgebra $b\A=\{ba\mid a\in\A\}$
is not nice in the relative algebra $\B\uhr b$.

\begin{prp}
\label{prp:local_nice}
Let $\B$ be a $\kappa$-Souslin algebra and
$\A$ a Souslin subalgebra of $\B$.
Let $b:=\sum\{x\in\B\mid x\A\text{ is nice in }\B\uhr x\}$.
Then $b\A$ is nice in $\B\uhr b$ and $(-b)\A$ is nowhere
nice in $\B\uhr(-b)$.
\end{prp}
\begin{proof}
It follows directly from the definitions
that $(-b)\A$ is nowhere nice.

Clearly, the property ``$x\A$ is nice in $\B\uhr x$''
descends from $x$ to $y\leq_\B x$.
We prove that this 
property is also preserved under taking arbitrary sums.
So let $M$ be a subset of $\B$, such that all elements of
$M$ have this property.
We want to show that for $x:=\sum M$
the subalgebra $x\A$ is nice in $\B\uhr x$.
We can without loss of generality assume that $M$ is an antichain.
Then $M$ is of cardinality $<\kappa$.
Furthermore we can assume that also $h''M$ is an antichain
by the argument used at the beginning of the proof of
the Representation Lemma \ref{lm:representation}.
We finally assume
that there is a Souslinization $T$ of $B$ such that
$M$ is a subset of $T_1$, the first nontrivial level of $T$,
and $T$ carries a t.e.r. $\equiv$ which represents $\A$.

Now for every element $r$ of $M$ there is
by part c) of the Representation Lemma \ref{lm:representation}
a club $C_r$ of $\kappa$,
such that $\equiv$ is nice on $T(r)\uhr C_r$.
Let $C$ be the club intersection of all sets $C_r$ for $r\in M$.
We claim that $\equiv$ is nice on the subtree
$$S=\bigoplus_{r\in M}T(r)\uhr C$$
of $T\uhr C$.
So let $s\equiv t$ in $S$ and $s'>s$.
If there is a unique member $r$ of $M$ below both nodes $s$ and $t$,
then we can directly apply the hypothesis on $r$.
Otherwise we would still have $r_s:=s\uhr 1\equiv t\uhr 1=:r_t$
and $h(r_s)=h(r_t)$ by our assumption that $h"M$ is an antichain.
But then we have that $r_th(s')>0$.
So there is a node $t^*>r_t$ equivalent to $s'$.
Finally, by niceness above $r_t$, there also is
a node $t'$ above $t$ such that $t'\equiv t^*\equiv s'$.
\end{proof}

\section{Large subalgebras}
\label{sec:large}

Large subalgebras can be regarded as the simplest type of
subalgebras\footnote{In \cite[pp.266]{jech}
such subalgebras are called ``locally equal''
and studies in the general context of
forcing with complete Boolean algebras.}.
They are closely related to symmetries of the Souslin algebra
and admit a detailed yet clear representation.

\begin{defi}
\label{defi:large}
Let $B$ be a complete Boolean algebra.
We say that $C$ is a \emph{large subalgebra of} $B$,
if there is an antichain $M$ of $B$,
such that $\langle A\cup M\rangle^\cm=B$.
We say that a large subalgebra $A$ of $B$ is \emph{$\mu$-large} for some
cardinal $\mu$ if there is an antichain $M$ of size $\mu$ such that $\langle A\cup M\rangle^\cm=B$.
\end{defi}


Note that large subalgebras of $\kappa$-Souslin algebras
are always atomless and therefore Souslin subalgebras,
since for every atom $a$ of $\A$,
the set $M\cup \{a\}$ of size $<\kappa$
would have to generate the relative algebra $\B\uhr a$.
But this is impossible,
because $\langle M\rangle^\cm$ is itself atomic.

As a first example we consider a $\kappa$-Souslin algebra that has
exactly one non-trivial subalgebra, and this subalgebra is large.

\begin{expl}
\label{expl:only1subalg}
Let $\B$ be a simple $\kappa$-Souslin algebra, i.e.,
that $\B$ has no proper atomless and complete subalgebra,
cf. Example \ref{expl:simple}.

We claim that the $\kappa$-Souslin algebra
$\C:=\B\times\B$ has exactly one proper atomless and complete
subalgebra, which is furthermore 1-large in $\C$.

Clearly, $\C$ has the large subalgebra
$$\A:=\{(b,b)\mid b\in\B\},$$
and $\A$ is 1-large in $\C$, because
$\langle \A\cup\{(1,0)\}\rangle^\cm = \C$.
As we have $\A\cong\B$, there are no (atomless and complete)
subalgebras of $\C$ below $\A$.

On the other hand we have
$$\C\uhr(1,0)\cong\C\uhr(0,1)\cong\B.$$
So if there was any other atomless and
complete subalgebra $\A'$ of $\C$,
then $(0,1)\cdot\A'$ or $(1,0)\cdot\A'$
would be a nontrivial subalgebra of the
respective relative algebra of $\C$.
But the latter are simple.
So the existence of such a subalgebra $\A'$ is impossible.
\end{expl}

In general, ($2^{\aleph_0}$-)large subalgebras always occur
whenever a $\kappa$-Souslin algebra has non-trivial symmetries.

\begin{thm}\label{thm:FPlarge}
Let $\B$ be a $\kappa$-Souslin algebra and $\varphi\in\Aut\B$.
Then the set of fixed points of $\varphi$ is a large subalgebra
$\A$ of $\B$.
In particular, if $\kappa > 2^{\aleph_0}$ then $\A$ is $2^{\aleph_0}$-large.
\end{thm}

\begin{proof}
We use Frol\'ik's Theorem, a deep result from
the theory of complete Boolean algebras
(cf. \cite[Theorem 13.23]{koppelberg}):
For every automorphism $f$ of a complete Boolean algebra $A$,
there is a partition of unity $\{a_0,a_1,a_2,a_3\}$ in $A$
such that $f\uhr(A\uhr a_0)$ is the identity and
for $i>0$ we have $f(a_i)\cdot a_i=0$.

We consider the at most countable family
$(\varphi^n\mid n\in\mathbb{Z})$ of automorphisms of $\B$
and let $(a_{n0},a_{n1},a_{n2},a_{n3})$ be
a partition of unity given by Frol\'ik's Theorem
for $\varphi^n$, $n\in\mathbb{Z}$.
Let $M$ be the set of atoms of the complete subalgebra of $\B$
that is (completely) generated
by the elements $\varphi^{k}(a_{ni})$ for $k,n\in\mathbb{Z}$ and $i<4$.
Then $M$ has by distributivity of $\B$ at most $2^{\aleph_0}$ elements.
Note that $\varphi\uhr M$ is a permutation of $M$
and if for some $x\in M$ and $n\in\omega$
we have $\varphi^n(x)=x$,
then the restriction of $\varphi^n$ to $\B\uhr x$ is the identity map.

We claim that $\langle\A\cup M\rangle^\cm=\B$.
Since $M$ is an antichain, it suffices to show
that for all $x\in M$ and $b\in\B\uhr x$
there is a member $a\in\A$, i.e., a fixed point of $\varphi$, with $ax=b$.
For all integers $n$ we know
that either $\varphi^n(b)=b$ or
$\varphi^n(b)$ is disjoint from $x$.
Let $a=\sum\{\varphi^n(b)\mid n\in\mathbb{Z}\}$ and it is easy to check that
the proof is finished.
\end{proof}
Note that the algebra $\C$ from Example \ref{expl:only1subalg} has exactly
two automorphisms: the identical mapping and flipping of coordinates.

The following  technical lemma states the existence of
\emph{optimal witnesses of largeness}.
With these witnesses at hand
we can easily deduce the main structural properties of large subalgebras.
\begin{lm}
\label{lm:dec_lar_sub}
Let $\A$ be a $\mu$-large subalgebra of the $\kappa$-Souslin algebra $\B$.
Define $X:=\{x\in\B\mid\B\uhr x = x\A\}$.
\begin{itemize}
\item[a)] The set $X$ is dense in $\B$, and $x<_\B y\in X$ imply $x\in X$.
\item[b)] For every $x\in X$ the restriction of the canonical projection
$h:\B\to\A$ to $\B\uhr x$, i.e. the map
$$ \varphi:\B\uhr x \to \A\uhr h(x),\quad b\mapsto h(b)$$
is an isomorphism between $\B\uhr x$ and $\A\uhr h(x)$.
The inverse map of $h\uhr(\B\uhr x)$ is given by multiplication with $x$:
$$\varphi^{-1}:\A\uhr h(x)\to \B\uhr x,\quad a\mapsto ax.$$
\item[c)] Every subset $M\subseteq X$ with $\sum M=1$ (or even $1-\sum M\in X$)
witnesses that $\A$ is large.
\item[d)] For every $x\in X$ there is a maximal element $y$ of $X$ above $x$.
\end{itemize}
If additionaly $\A\uhr a\neq\B\uhr a$ for all $a\in\A^+$ and
$Y$ denotes the set of maximal elements of $X$, then the following hold as well.
\begin{itemize}
\item[e)] The image of $Y$ under $h$ is a maximal antichain of $\A$.
\item[f)] Every set of pairwise disjoint elements of $Y$
is extendible to a maximal antichain $\subset Y$ of $\B$.
\item[g)] For every maximal antichain $M\subseteq Y$ we have $h''M=h''Y$.
\end{itemize}
\end{lm}
The announced optimal witnesses of largeness are simply the partitions of
unity that are subsets of the set $Y$ defined in the lemma.
\begin{proof}
We only give proofs of points c-e).
The rest is then trivial or follows by standard arguments.

For the proof of c) pick a subset $M\subset X$ with $\sum M=1$.
We want to show that every $b\in\B^+$ is of the form
$$b'=\sum\{xh(bx)\mid \,\, x\in M,\,\, xb>_\B 0\}.$$
It is clear that $b'\geq_\B b$, because $h(b),\sum M\geq b$.
On the other hand we conclude from part b)
that $xh(bx)=bx$ for $x\in X$, so $b'\leq_\B b$ as well.
So we have $\langle\A\cup M\rangle^\cm=\B$.

To prove the existence of maximal elements of $X$,
it is enough to verify that
$X$ is closed under taking sums over increasing sequences of length $<\kappa$.
So let $x_\alpha\in X$ and $x_{\alpha+1}>_\B x_\alpha$ for all
$\alpha<\delta\,(<\kappa)$.
Set $x=\sum x_\alpha$.
We prove that every $x_\alpha$ is in $x\A$ as follows.
Fix $\alpha$.
For every $\beta>\alpha$ pick an element $a_\beta\in\A$ that satisfies
$x_\beta a_\beta=x_\alpha$.
Defining $a=\prod_{\beta>\alpha} a_\beta$
we get $x_\beta a=x_\alpha$ for all $\beta>\alpha$
and therefore
(using the infinite distributive law available in $\B$)
$$xa=\sum x_\beta a=x_\alpha .$$
But then we already have $x\A = \B\uhr x$,
because every element $y\in\B\uhr x$
can be decomposed into a sum
$$y=\sum_{\alpha<\nu} y_\alpha\quad
\text{ with }y_\alpha:=y(x_{\alpha+1}-x_\alpha).$$
By the same argument as above we have
$y_\alpha\in x\A$ for all $\alpha<\delta$.

Concerning the proof of part e) of the lemma,
we know by a) and d) that $1=\sum Y$ and therefore $1=\sum h''Y$.
It remains to show
that for all pairs $x,y\in Y$ with $h(x)h(y)>_\B 0$ we have $h(x)=h(y)$.
To reach a contradiction
we assume the existence of a pair $x,y\in Y$ with a non-empty intersection
of the $h$-images, $h(x)h(y)>_\B 0$, yet $h(x)-h(y)>_\B 0$.
This implies $x-h(y)>_\B 0$,
for otherwise $h(x)h(y)=0$.
We set
$$z:=y+(x-h(y))$$
and get that $z >_\B y$ and $zh(y) = y$.
This shows that $y,\,z-y \in z\A^+$ and implies thus that
$z\in X$ (because $z-y<_\B x$ so $z-y\in X$),
contradicting the maximality of $y$ in $X$.
\end{proof}

We are prepared to state and prove the key properties of large subalgebras.
Figure \ref{fig:dec_large} below
corresponds to part c) of the theorem in terms of Souslinizations
and illustrates the strong resemblance
between superalgebra and large subalgebra.

\begin{figure}[b!]
\begin{center}
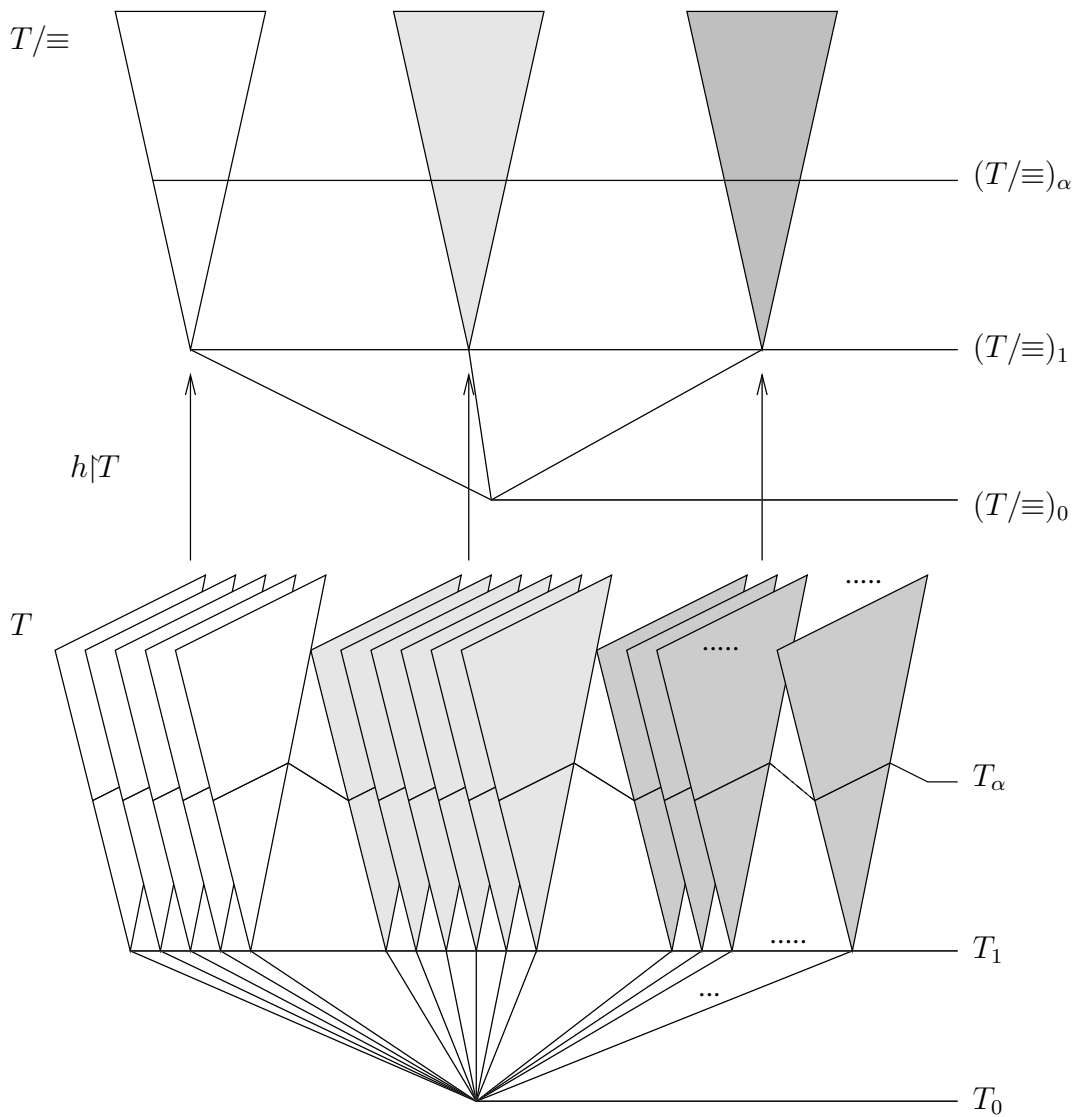
\caption{Nice representation
of a large subalgebra $\A=\RO(T/\!\!\equiv)$ (above)
of $\B=\RO T$ (below).
The algebras and their Souslinizations
can be decomposed in relative parts (here there are $3$),
such that the $\B$-part consists of a Cartesian product (resp. a tree sum)
of the corresponding $\A$-part.
The factors/tree summands of the $\B$-part are interconnected by
the isomorphism coming from the projection $h$}
\label{fig:dec_large}
\end{center}
\end{figure}

\begin{thm}
\label{thm:large}
Let $\A$ be a large subalgebra of the $\kappa$-Souslin algebra $\B$.
Then the following hold.
\begin{enumerate}[a)]
\item $\A$ is a nice subalgebra of $\B$.
\item There is a group $G$ of size less than $\kappa$
of automorphisms of $\B$ such that
$$\A=\{x\in\B\mid\,(\forall \varphi\in G)\,\varphi(x)=x\}.$$
If furthermore $\mu$ is the minimal cardinal such that
$\A$ is $\mu$-large in $\B$,
then $G$ can be chosen of size $\aleph_0 \cdot \mu$.
\item There are a maximal antichain $N$ in $\A$ and a map $f$ associating a
  cardinal $f(a)$ to each member $a$ of $N$ such that 
we have the following representation of $\B$ over $\A$:
$$\B\cong\prod_{a\in N}(\A\uhr a)^{f(a)}.$$
\item If $\B$ is homogeneous, then $\A$ and $\B$ are isomorphic.
\end{enumerate}
\end{thm}

\begin{proof}
Let $h:\B\to\A$ be the canonical projection. To prove a) we 
construct a Souslinization of $T$ such that $h''T$ souslinizes $\A$.
Let $T_1$, the first non-trivial level of $T$,
be a maximal antichain of $\B$ consisting of maximal
elements $x$ with $x\A=\B\uhr x$,
i.e., $T_1\subset Y$ in the notation used above.
Then $N:=h''T_1$ is an antichain in $\A$.
Now fix a pre-image $b_a\in h^{-1}(a)\cap T_1$ for each $a\in N$.
To construct the higher successor levels,
we first refine the nodes above $b_a$
for each $a\in N$ and then copy these refinements by virtue of the
isomorphisms
$$\psi_{b}:\B\uhr b_a\to\B\uhr b,\,x\mapsto bh(x)$$
for all $b\in h^{-1}(a)\cap T_1$.
This automatically transfers to limit levels and guarantees
that also for limit $\alpha$ the set $h''T_\alpha$
consists of products over cofinal branches in $T\uhr\alpha$.

Finally, in order to prove that the relation
$$s\equiv t:\iff \hgt s=\hgt t\text{ and }h(s)=h(t)$$
is a nice t.e.r., let $s\equiv t$ on level $T_\alpha$
and let $s'$ be a $T$-successor of $s$.
Then the node $t'=(t\uhr 1)\cdot h(s')$
is the witness for this instance of niceness.


For the proof of b) let $A\subset Y$ be a maximal antichain in $\B$
such that
$\langle \A \cup A \rangle^\cm = \B$.
For $a,b \in A$ with $h(a) = h(b)$ let
$$\varphi_{ab}:\B\to\B,\quad
x \mapsto x - (a+b) + a \cdot h(bx) + b \cdot h(ax),$$
which is a self-inverse automorphism of $\B$
interchanging $a$ with $b$.
The fixed points of $\varphi_{ab}$ form the subalgebra
$$\B\uhr(-(a+b)) \cup (a+b)\A.$$
Letting $G$ be the group
of automorphisms of $\B$ generated by the set
$$\{ \varphi_{ab}\mid\ a,b \in A,\, h(a) = h(b) \}$$
we see that this is as stated in the theorem.

Part d) readily follows from part c) which we prove now.
Let $A\subset Y$ be as above and set $N := h"A = h"Y$.
Define $f:N \to \kappa$ by $f(a) := |h^{-1}(a) \cap A|$.
Taking into account that for each $b \in a$ we have that
$$\B\uhr b = b\A \cong \A \uhr h(b)$$
we get as a Cartesian product
$$\B \cong \prod_{b \in A} \B \uhr b \cong \prod_{b \in A} \A \uhr h(b)
\cong \prod_{a \in N} (\A \uhr a)^{f(a)}.$$
This finishes the proof.
\end{proof}

\section{Nowhere large subalgebras}
\label{sec:nowhere_large}

We now consider more general algebras
with more involved representation features.

\begin{defi}\label{defi:inice}
Let $\B$ be a $\kappa$-Souslin algebra,
$T$ be a Souslinization of $\B$ and $\A$ a complete subalgebra of $\B$.
\begin{enumerate}[a)]
\item $\A$ is \emph{nowhere large} (in $\B$) if for all $b\in\B^+$
we have $b\A\neq\B\uhr b$.
\item A t.e.r. $\equiv$ on $T$ is \emph{$\mu$-nice}
(for a cardinal $\mu<\kappa$)
if it is nice and for all $\alpha<\beta<\hgt(T)$ and
$$(\forall r\in (s\uhr\alpha)/\!\!\equiv)\quad
|\{t\in s/\!\!\equiv\mid t\uhr\alpha=r\}| \, \geq \, \mu \, ,$$
i.e., 
for all $s\in T_\beta$,
the projections $t\mapsto t\uhr\alpha$,
when restricted to the $\equiv$-class of $s$,
are $(\geq\!\mu)$-to-one.
\item $\A$ is \emph{$\infty$-nice} in $\B$ if
for one/any cardinal $1<\mu<\kappa$ there is a club $C$ of $\kappa$,
such that $T\uhr C$ carries
an $\mu$-nice t.e.r. $\equiv$ that represents $\A$.
\end{enumerate}
\end{defi}

Note that in point c) one cardinal $\mu$ suffices as
an easy argument shows that a $2$-nice t.e.r. on a $\kappa$-Souslin tree
can be turned into a $\mu$-nice t.e.r. for any $\mu<\kappa$
by concentrating the tree on a club set of levels.

\begin{rem}
\label{rem:local_large}
If $\A$ is any atomless complete subalgebra of the $\kappa$-Souslin algebra
$\B$, then obviously we have for 
$$ x := \sum\{ b \in \B\,\mid\, b\A = \B\uhr b \}$$
that $x\A$ is large in $\B \uhr x$
while $(-x)\A$ is nowhere large in $\B\uhr(-x)$.
This corresponds to the situation for niceness as stated in
Proposition \ref{prp:local_nice}.
Note that, if a nice t.e.r. $\equiv$ for the large portion $x\A$
as in the last section is found,
then the $\equiv$-classes on limit levels are discrete (above $x$).
\end{rem}

Before we give a first example of an $\infty$-nice subalgebra,
we turn to clarify the interrelationship between the new notions.

\begin{prp}
Let $\A$ be a nice subalgebra of the $\kappa$-Souslin algebra $\B$.
Then $\A$ is $\infty$-nice if and only if it is nowhere large.
\end{prp}
Plainly:
the $\infty$-nice subalgebras are just the nice and nowhere large ones.
\begin{proof}
Let $T$ souslinize $\B$,
and let the nice t.e.r. $\equiv$ on $T$ represent $\A$.
We start from left to right, so let $\equiv$ be $2$-nice.
We show for every node $s$ of $T$, that $\A\cdot s\neq\B\uhr s$.
Pick any node $t$ above $s$ in $T$. Since $\equiv$ is $2$-nice, there
is a node $r$ above $s$ and equivalent to $t$, so $t\not\in\A\cdot s$.
So what we actually have shown, is $\A\cdot s=\{0,s\}$.

For the other implication let $\A$ be nowhere large.
We define a club set $C\subset\kappa$,
such that the restriction of $\equiv$ to $T\uhr C$ is $2$-nice.
The inductive construction of $C$ is straightforward
once we have proven the following claim.\\
{\bf Claim.}
Given any $\alpha<\kappa$ there is a $\beta\in(\alpha,\kappa)$ such that
for all nodes $t\in T_\beta$ there is a node $t'\in T_\beta\setminus\{t\}$ above
$t\uhr\alpha$ and equivalent to $t$.

To prove the claim by contradiction,
assume that there is an ordinal $\alpha$
such that for all $\beta>\alpha$
there is a node $t^\beta$ of level $T_\beta$ such that
above its predecessor $t^\beta\uhr\alpha$
there is no other node equivalent to $t^\beta$.
By the pigeon hole principle one of the $<\kappa$ many nodes of level
$T_\alpha$ sits underneath $\kappa$ many of these nodes $t^\beta$.
So we can assume,
that we have one node $s^*\in T_\alpha$
such that $s^*<t^\beta$ for all $\beta>\alpha$.
But then these nodes $t^\beta$ for $\beta>\alpha$ span a tree of
height $\kappa$ which by the Subtree Lemma \ref{lm:subtree}
contains a canonical subtree $\{s<r\}\cup T(r)$ of $T$
for some node $r$ above $s^*$.
But then in turn we have that $\A\cdot r=\B\uhr r$,
which contradicts the hypothesis
on $\A$ to be nowhere large in $\B$.
\end{proof}

The basic example we consider now can easily be generalized
to $\kappa$-Souslin algebras for regular $\kappa>\aleph_0$.

\begin{expl}
\label{expl:prod_nice}
Let $S$ and $T$ be $\aleph_0$-splitting
$\omega_1$-trees such that their tree product $S\otimes T$ is
$\omega_1$-Souslin\footnote{For example, the principle $\dia$ implies,
that for every given $\omega_1$-Souslin tree $S$
there is an $\omega_1$-Souslin tree $T$, such that
$S\otimes T$ is c.c.c..
For a proof of this fact see \cite[Lemma 7.3]{larson}.}.
Set $\B:=\RO(S\otimes T)$
and $(s,t)\sim(u,v)$ if and only if $s=u$.
Then $\sim$ is an $\aleph_0$-nice t.e.r.:
If $s<_S s'$ and $\hgt_S(s)=\hgt_T(t)=\hgt_T(r)$ and $t'>_T t$,
then for any $r'>r$ we have that $(s',t')\sim(s',r')$.
So $\sim$ is nice.
The $\aleph_0$-part follows from the splitting assumption on $T$.
The quotient tree $(S\otimes T)/\!\!\sim$ is obviously isomorphic to $S$,
and the subalgebra $\A$ represented by $\sim$ is $\infty$-nice in $\B$.
\end{expl}
\begin{rem}
Note that not not all $\infty$-nice subalgebras of $\kappa$-Souslin algebras
do have a complement as in the example above.
For example, one of the subalgebras, that will be constructed in Section
\ref{sec:no_schroeder}, call it $\C$, 
is $\infty$-nice, yet isomorphic to the superalgebra $\B$.
If there was a subalgebra $\C'$ of $\B$ independent of $\C$,
then an isomorphic copy of $\C'$ would exist inside of $\C$.
This contradicts the chain condition satisfied by $\B$.
\end{rem}

\subsection{Homogeneity and $\infty$-nice subalgebras}

Recall that a Boolean algebra $B$ is \emph{homogeneous}, if for all pairs
$b\in B^+$ there is a Boolean isomorphism between $B$ and $B\uhr b$,
while homogeneity of the tree $T$ means that for all pairs of nodes $s,t$
of the same height in $T$ the trees $T(s)$ and $T(t)$ of nodes
above $s$ and $t$ respectively are isomorphic

\begin{prp}\label{prp:hom_souslinization}
Let $\kappa$ be an uncountable, regular cardinal.
Then every homogeneous $\kappa$-Souslin algebra
has a homogeneous Souslinization.
\end{prp}
\begin{proof}
Let $\B$ be homogeneous and $T$ be any Souslinization of $\B$.
Our task is to find a club $C\subset\kappa$ such
that $T\uhr C$ is a homogeneous $\kappa$-Souslin tree.
By the homogeneity of $\B$ we can choose
a Boolean isomorphism $\psi_{st}:\B\uhr s\to\B\uhr t$
for every pair $s,t\in T$ of the same height $\alpha<\kappa$.
By the Restriction Lemma \ref{lm:restriction}
for Isomorphisms, 
there is also a club $C_{st}\subset\kappa$ containing  $\alpha$,
such that $\psi_{st}\uhr(T(s)\uhr C_{st})$
is an isomorphism onto $T(t)\uhr C_{st}$.

Finally, we define $C$ to be the range
of the normal sequence $(\gamma_\nu)$ which is given as follows:
Set $\gamma_0=0$ and let for $\nu<\kappa$
$$\gamma_{\nu+1}:=\min\bigcap_{s,t\in T_{\gamma_\nu}}C_{st}\setminus(\gamma_\nu +1),$$
the limit values of the sequence are then determined by normality.
\end{proof}

In Section \ref{sec:large} we have seen that the existence of large
subalgebras is linked to the existence of automorphisms.
Yet if there are enough automorphisms,
which here means: if $\B$ is homogeneous,
then we even get subalgebras of different kinds
($\infty$-nice and nowhere nice,
see also Theorem \ref{thm:exist_non-nice}).

\begin{thm}\label{thm:subalg_of_hom}
Every homogeneous $\kappa$-Souslin algebra has
an $\infty$-nice subalgebra.
\end{thm}
\begin{proof}
Let $T$ be a homogeneous $\kappa$-Souslin tree,
i.e., for every pair $s,t$ of nodes on the same level of $T$
there is a tree isomorphism between $T(s)$ and $T(t)$.
We inductively show for $\alpha<\kappa$ that $T\uhr\alpha$
carries an $2$-nice t.e.r. $\equiv$ using the homogeneity of $T$.
After construction stage $\alpha<\kappa$ we will have fixed
the t.e.r. $\equiv$ on $T\uhr\alpha+1$,
sets $I_\gamma\subset T_\gamma$ of representatives
of the $\equiv$-classes for $\gamma\leq\alpha$ and
a family of isomorphisms
$$\{\varphi_{st}:T(s)\cong T(t)\mid \,
s\equiv t,\,\hgt(s)\leq\alpha\}.$$
These isomorphisms commute in the sense that
$$\varphi_{tt}=\id_{T(t)}
\quad\text{ and }\quad
\varphi_{st}=\varphi_{rt}\circ\varphi_{sr}
\quad\text{ (for all } r\equiv_\gamma s\equiv_\gamma t).$$
Furthermore they
have the following coherence property:
for $s,t\in T_\alpha$ and $r=s\uhr\gamma$, $u=t\uhr\gamma$ where
$\gamma<\alpha$ and $\varphi_{ru}(s)=t$
we have $\varphi_{st}=\varphi_{ru}\uhr T(s)$.
These isomorphisms $\varphi_{st}$ will help to guarantee
that $\equiv$ always remains honest.

We will use the representatives from the set $I_\alpha$
for the constructions of both the t.e.r. and the tree isomorphisms.
We will first define the relevant piece of structure
above the representative nodes $r\in I_\alpha$
and then copy it over to the equivalent nodes by virtue of the 
tree isomorphisms that have already been fixed.

In the case of the successor ordinal $\alpha+1$,
we consider the equivalence relation $\equiv$ on $T_\alpha$,
the set of representatives $I_\alpha\subset T_\alpha$ and
the isomorphisms $\varphi_{st}$ for $s\equiv_\alpha t$,
all given by the inductive hypothesis.
For $s\in T_\alpha$ denote by $r_s$
the unique element of $s/\!\!\equiv\cap\, I_\alpha$.
In order to define $\equiv$ on $T_{\alpha+1}$,
we first choose for each $r\in I_\alpha$ a partition
of $\NF(r)$ into $2$ sets $P^r_0,P^r_1$ of equal cardinality.

Then for all $s,t\in T_{\alpha+1}$ 
we let $s$ and $t$ be equivalent
if their (immediate) predecessors $s^-$ and $t^-$
are and if their
images under the tree isomorphisms sending them above
the representative node $r=r_{s^-}=r_{t^-}$
lie in the same member of the partition, both in $P^r_0$ or both in $P^r_1$:
$$s\equiv t \,:\iff \,s^-\equiv t^- \text{ and }
(\exists i\in \{0,1\})\,\,
\varphi_{s^-r}(u),\,\varphi_{t^-r}(v) \in \,P^{r}_i.$$
Afterwards,
we pick a set of representatives $I_{\alpha+1}\subset\bigcup_{r\in I_\alpha}\NF(r)$.

Finally,
we have to choose the tree isomorphisms $\varphi_{st}$
for all equivalent pairs $s,t\in T_{\alpha+1}$
such that the coherence requirement as formulated above is satisfied.
Fix a representative $r\in I_{\alpha+1}$ and
choose for a pair of successors $(s,t)$ of $r^-:=r\uhr\alpha$,
both equivalent but unequal to $r$,
isomorphisms $\varphi_{st}$ and $\varphi_{sr}$ respectively and
let $\varphi_{st}=\varphi_{rt}\circ\varphi_{rs}^{-1}$.
For $s,t$ both equivalent to $r$,
but not necessarily successors of $r^-$, define
$$\varphi_{st}:=
(\varphi_{r^-t^-}\uhr T(v))\circ\varphi_{uv}\circ(\varphi_{s^-r^-}\uhr T(s)),$$
where $u:=\varphi_{s^-r^-}(s)$ and $v:=\varphi_{t^-r^-}(t)$.

Whenever $\alpha<\kappa$ is a limit ordinal,
we have no choice for the equivalence relation $\equiv$ on $T_\alpha$:
For $s,t\in T_\alpha$ we let $s\equiv t$
if and only if
$s\uhr\gamma\equiv t\uhr\gamma$ for all $\gamma<\alpha$.

Before defining the remaining tree isomorphisms we check,
that this definition yields a nice t.e.r. up to level $T_\alpha$ .
So fix $s\in T_\alpha$.
For every $\gamma<\alpha$ and $u\equiv s\uhr\gamma$
there is some $t\in T_\alpha$ equivalent to $s$ and above $u$,
namely $t=\varphi_{s\upharpoonright\gamma,u}(s)$.
So niceness is maintained up to level $\alpha$,
and for equivalent pairs $(s,t)$ of this kind we already have the isomorphism
$\varphi_{st}=\varphi_{s\upharpoonright\gamma,t\upharpoonright\gamma}\uhr T(s)$
at hand.
But there can be equivalent nodes $s$ and $t$ on level $\alpha$,
such that for all their pairs $u,v$ of respective
predecessors on the same level we have $\varphi_{uv}(s)\neq t$.
However, each $\equiv$-class $r/\!\!\equiv$
divides into a partition $\mathscr{P}$
such that for every pair of nodes $s,t\equiv r$,
both $s$ and $t$ are elements of the same member of $\mathscr{P}$
if and only if they have such an inherited isomorphism
$\varphi_{st}=\varphi_{s\upharpoonright\gamma,t\upharpoonright\gamma}\uhr T(s)$.

After choosing a set of representatives $J$
for the partition $\mathscr{P}$ 
and fixing isomorphisms $\varphi_{st}$ for representatives $s,t\in J$
we can construct the still missing isomorphisms in the same manner as above.

We finally choose a set $I_\alpha$ of representatives
for the $\equiv$-classes of the limit level
$T_\alpha$ without any further restriction.

This finishes the construction of $\equiv$, and we hope
that it is clear that the result is
a $2$-nice t.e.r.\ on $T$.
\end{proof}

\subsection{Hidden symmetries}
\label{sec:hidden_symm}
In Section \ref{sec:rigid-non-rigid} we will construct a $\kappa$-Souslin
algebra with an $\infty$-nice subalgebra but without large subalgebras,
i.e., without automorphisms except for the identity.
The next lemma and the subsequent theorem say
that in such a situation there have to be other subalgebras,
in particular subalgebras which do have symmetries.

This stands in sharp contrast to the subalgebra to be constructed
in Section \ref{sec:unique_non-nice} which is essentially a unique subalgebra.

\begin{lm}\label{lm:aut_over_inice}
Let $\A$ be an $\infty$-nice subalgebra of a $\kappa$-Souslin algebra $\B$.
Then there is an $\infty$-nice subalgebra $\C$ of $\B$,
such that $\C$ admits a non-trivial automorphism $\varphi$
and $\A$ is the subalgebra of $\C$
that consists of the fixed points under $\varphi$.
So $\A$ is a large subalgebra of $\C$, yet for all $a\in \A^+$ we have
$\A\uhr a \neq \C\uhr a$.
\end{lm}
\begin{proof}
Let $T$ souslinize $\B$ and
let $\equiv$ on $T$ be $\infty$-nice and represent $\A$ in $T$.
Choose any limit $\lambda<\kappa$ and
let $\simeq$ coincide with $\equiv$ on $T\uhr(\lambda+1)$.
Now we divide every $\simeq$-class $a$ of level $T_\lambda$
in two indexed parts $a=a_0\dot{\cup}a_1$,
such that for every pair $s\in a_0$ and $r<_T s$
there is a node $t\in a_1$ above $r$
and vice versa, i.e.,
for $r<t\in a_1$ there exists $s\in a_0\cap T(r)$.
Another way to formulate this condition is to say that we consider
$T_\lambda$ as a subspace of $[T\uhr \lambda]$ and require that
each class $a$ divided in two parts that lie densely in $a$.
This can be done after choosing an enumeration of minimal length
of the predecessor set $\bigcup\{s\in T\mid (\exists t\in a)s<_T t\}$.
The family of these partitions gives rise to a map $i:T_\lambda\to \{0,1\}$,
associating to every node $s$ the index $i(s)$ with $s\in a_{i(s)}$.

Now let for $\alpha>\lambda$ and $s,t\in T_\alpha$
$$s\simeq t\,\,:\iff \, s\equiv t
\,\text{ and }\,i(s\uhr\lambda)=i(t\uhr\lambda).$$
Then $\simeq$ is clearly $\infty$-nice when restricted to $T\uhr C$,
where $C=\{0\}\cup\kappa\setminus\lambda+1$.
This shows that the subalgebra
$\C:=\langle \sum s/\!\!\simeq\mid s\in T\rangle^\cm$
is nice and nowhere large in $\B$.

Furthermore, for every $s\in T$ above level $\lambda+1$
the $\equiv$-class of $s$ is divided into exactly two $\simeq$-classes.
So we can define the automorphism $\varphi$ of $(T\uhr C)/\!\!\simeq$
that for each
$\equiv$-class interchanges the two $\simeq$-classes.
Then $\varphi$ naturally extends to an automorphisms of $\C$ that
has $\A$ as its fixed point algebra which is
by Theorem~\ref{thm:FPlarge} large in $\C$.
(In fact $\A$ is 1-large in $\C$ as witnessed by $\{\sum f^{-1}"\{0\}\}$.)
\end{proof}

\subsection{Nowhere nice subalgebras}

The main idea of the last proof,
that of dividing the classes on a limit level in dense subsets,
can be also used to construct nowhere nice subalgebras.

\begin{thm}\label{thm:exist_non-nice}
If a $\kappa$-Souslin algebra $\B$ has a nice and nowhere large subalgebra $\A$
then there is a nowhere nice subalgebra $\C$ of $\B$
and $\A$ is an $\infty$-nice subalgebra of $\C$.
\end{thm}

\begin{proof}
Let $T$ souslinize $\B$ and let $\equiv$ represent $\A$ in $T$.
We inductively construct an almost nice, yet not nice
refinement $\simeq$ of $\equiv$,
which represents $\C$ as stated in the theorem.
Up to level $T_{\omega}$ the new relation coincides with $\equiv$.
Limit levels have to be treated canonically,
and on double successor steps $\alpha$
as well as on successors $\alpha$ of limits with uncountable cofinality,
we choose the minimal possible refinement by meeting
$$s\simeq t:\iff s\equiv t\text{ and }s^-\simeq t^-.$$

Let now $\alpha < \kappa $ be a limit
of countable cofinality, $\cf\alpha=\aleph_0$.
Note that the (induced) $\simeq$-classes on the space
$[T\uhr\alpha]$ form closed subsets without isolated points.
Therefore, regarding $T_\alpha$ as a subspace of $[T\uhr\alpha]$, the former
also divides into a partition whose members are closed subsets without
isolated points (as $T_\alpha$ is suitable for $\simeq$ on $T\uhr\alpha$).
To define $\simeq$ on level $T_{\alpha+1}$ 
we first refine $\simeq$ on $T_\alpha$ to the equivalence relation $\sim$
in a way such that every $\simeq$ class splits in two
$\sim$-classes and for every $s\simeq t\in T_\alpha$ and $u<_T t$,
there is a successor of $u$ in $s/\!\!\sim$.
(As in the proof of Lemma \ref{lm:aut_over_inice},
one could also say that the $\sim$-classes lie densely in the sense
of $[T\uhr\alpha]$ in the $\simeq$-classes.)
Then let for $s,t\in T_{\alpha+1}$:
$$s\simeq t\,\,:\iff
\,s\equiv t\,\text{ and }\,s\uhr\alpha\sim t\uhr\alpha.$$
This procedure clearly refines $\equiv$ to an almost nice t.e.r. $\simeq$.

Next we show, that no Souslinization $S$ of $\B$ admits a nice
t.e.r.\ representing $\C$.
By the Restriction Lemma \ref{lm:restriction}
we only need to consider restrictions $S=T\uhr C$ of $T$
to a club $C\subset\kappa$.
So let $\alpha\in C$ be a limit ordinal of countable cofinality
and choose $s,t\in T_\alpha$,
such that $s\simeq t$ but $s\not\sim t$.
Then for every $r\in T\uhr C$ above $s$
there is no successor of $t$ which is $\simeq$-equivalent to $r$.
So $s,t,r$ witness that $\C$ is not nice.

If we now let $C$ be the set of all limit ordinals below $\kappa$ joined by 0,
and defining on $T/\!\!\simeq$ the t.e.r. $\approx$ by
$$(s/\!\!\simeq)\approx(t/\!\!\simeq):\iff s\equiv t,$$
then it is easy to see,
that the $\infty$-niceness descends from $\equiv$ to $\approx$. 
\end{proof}

\begin{rem}
\label{rem:non-nice}
By Theorem \ref{thm:exist_non-nice} and
since niceness and largeness are local properties,
if the $\kappa$-Souslin algebra $\B$
has a non-large subalgebra of $\B$,
then there is also one which is not nice.
In particular, if $\B$ is homogeneous or if $\B$ has a pair of independent
Souslin subalgebras, then $\B$ has nowhere nice subalgebras.
\end{rem}

\part{Some constructions of $\kappa$-Souslin algebras
with certain subalgebras}

\section{T.e.r.s and topology}

In this section we develop topological tools which we use 
to construct Souslin algebras with nowhere large subalgebras.
For these tools to be applicable also in cases where $\kappa>\aleph_1$
we have to generalize a few notions and facts concerning Baire Category.

We then formulate and prove the Reduction Lemma for t.e.r.s
which roughly states that (under favorable circumstances)
for a given tree of limit height with a t.e.r. on it, 
there is an extension of the tree such that the t.e.r. remains honest,
i.e. is preserved.

The constructions carried out in the subsequent sections
use $\dia$ hypotheses.
The Reduction Lemma and the surrounding lemmata can of course
also be applied in forcing constructions of generic $\kappa$-Souslin trees
once the hypothesis on cardinal arithmetic is satisfied in the ground model.

\subsection{Some basic descriptive set theory for weight $\mu$}
\label{sec:dst}

We introduce some variants of several classical topological notions
that we will use in the Souslin tree constructions in subsequent sections.
The spaces of interest are all homeomorphic to ${^\mu}\mu$, the analog of
Baire space $\mathscr{N}$ for some regular cardinal $\mu$.
Furthermore, the generalizations of some classical results as formulated here
only hold in case that $\mu{^{<\mu}} = \mu$.
(While in the case of $\mu=\aleph_0$ this follows from the axiom of choice,
it is an extra assumption extending {\sf ZFC + $\dia_{\mu^+}$}
if $\mu$ is uncountable.)

So in this section (and also in the remainder of the article)
$\mu$ denotes a regular cardinal satisfying $\mu{^{<\mu}} = \mu$.
Letting $\kappa := \mu^+$, the letter $T$ is used in this section for a
$\mu$-closed and $\kappa$-normal tree of height $<\kappa$
which has an isomorphic copy of ${^{<\mu}}\mu$ densely embedded
onto a club set of its levels.
Thus the spaces $[T]$ and ${^\mu}\mu$ are homeomorphic.

For a topological space $\mathscr{X}$
and a subset $M$ of $\mathscr{X}$ we say
that $M$ is $\mu$-$G_\delta$ if $M$ is the intersection
over a family of size $\mu$ of open subsets of $\mathscr{X}$.
The notion of $\mu$-$F_\sigma$ is defined analogously.
We start with the analog of the Baire Category Theorem.

\begin{thm}[Baire Category Theorem for weight $\mu$]
\label{thm:baire}
Assume that $\mu$ is a regular cardinal.
For each $\nu<\mu$ let $U_\nu$ be a dense open subset of ${^\mu\mu}$.
Then the intersection $\bigcap_{\nu<\mu} U_\mu$ is dense in ${^\mu\mu}$.
\end{thm}
\begin{proof}
First note that
the intersection of less than $\mu$ open subsets of ${^\mu\mu}$
is open.
From this it is easy to see, that $\bigcap_{\nu<\mu} U_\mu$ intersects every
non-empty open subset of ${^\mu\mu}$.
\end{proof}

Because of Theorem \ref{thm:baire},
we say that a set $M\subset[T]$ is $\mu$-\emph{comeagre} if it
contains a $\mu$-$G_\delta$ set
which is the intersection over a family of $\mu$ dense open sets,
and we say that it is $\mu$-\emph{meagre},
if its complement in $[T]$ is $\mu$-comeagre, i.e.,
if it is the union of a family of up to $\mu$ nowhere dense sets.
We furthermore call a topological space $\mathscr{X}$ \emph{$\mu$-Baire}
if every $\mu$-comeagre subset of $\mathscr{X}$ is dense in $\mathscr{X}$.
So e.g., the above theorem simply states
that ${^\mu\mu}$ is $\mu$-Baire.
On the other hand, every discrete space is $\mu$-Baire as well.

\begin{prp}
  \label{prp:baire_for_classes}
  Let $S$ be a $\mu$-closed $\kappa$-Souslin tree carrying the t.e.r. $\equiv$.
  Then there is a club $C \subset \kappa$ such that for every element
  $\alpha$ of $C$ with cofinality $\cf(\alpha)=\mu$
  letting $T=S\uhr\alpha$
  every $\equiv$-classe of $[T]$ is $\mu$-Baire
  (when equipped with the topology inherited from $[T]$).
\end{prp}
\begin{proof}
  Set $\B := \RO(S)$
  and let $\A$ denote the subalgebra represented by $\equiv$.
  Recall from Remark \ref{rem:local_large} that there is a unique element
  $b \in \B$ with the property that $b\A$ is nowhere large in $\B\uhr b$ and
  $(-b)\A$ is large $\B \uhr (-b)$.
  We will split up $S$ with the aid of $b$ and gain a decomposition
  of $[T\uhr\alpha]$ and its $\equiv$-classes in a discrete
  (and therefore $\mu$-Baire) part and
  and one that is $\mu$-Baire by resemblance to ${^\mu\mu}$.

  Let $Y$ be a set of optimal witnesses of largeness
  for $(-b)\A$ in $\B \uhr (-b)$
  as found in the proof of Lemma \ref{lm:dec_lar_sub}.
  Let $\beta<\kappa$ be large enough such that
  $Y\subset \langle S_\beta \rangle^\cm$,
  i.e. $X$ is contained in the subalgebra completely generated
  by the $\beta$th level of $S$.
  Now with
  $$G := \{s\in T_\beta \mid s\leq_\B -b\}$$
  and
  $$H := T_\beta \setminus G = \{s\in T_\beta \mid s\leq_\B b \}$$
  let
  $$S^+ := \bigoplus_{s\in G} T(s)$$
  and
  $$S^- := \bigoplus_{s\in H} T(s).$$
  In the part $S^+$ (where $\A$ is large)
  the limit classes of $\equiv$ are discrete subsets $[S^+\uhr\alpha]$.

  Now we consider the part $S^-$ (where $\A$ is nowhere large).
  Pick the club $C' \subset \kappa$ such that
  for all subsequent members $\gamma<\delta$ of $C'$
  and all pairs of nodes $s<t$ of $S^-$ 
  with  $\hgt_{S^-}(s)=\gamma$ and $\hgt_{S^-}(t)=\delta$
  there are $\mu$ further successors of $s$ equivalent to $t$,
  i.e.,
  on $S^-\uhr C$ the t.e.r. $\equiv$ ``splits'' $\mu$-ary
  immediately above every node.
  Our final club set is $C := \{\beta + \gamma \mid \gamma \in C'\}$.

  It is routine to check that for every member $\alpha$ of $C'$ with
  $\cf(\alpha)=\mu$ and $x\in [S^-\uhr\alpha]$ we have
  $(x/\!\!\equiv)\,\,\approx {^\mu\mu}$
  using the fact that $S$ is $\mu$-closed.
  But then for every $\alpha \in C$ of cofinality $\mu$
  and $x\in [S\uhr\alpha]$ the class $x/\!\!\equiv$
  is decomposed in a discrete $S^+$-part and a continuous $S^-$-part.
  If now for $\nu<\mu$ the set $U_\nu$ is open dense in $x/\!\!\equiv$,
  then it contains the whole discrete part
  and an open dense subset of the $S^-$-part.
  Then Theorem \ref{thm:baire} immediately states
  that the $\mu$-comeagre subset $\bigcap U_\nu$
  of $x/\!\!\equiv$ is dense.
\end{proof}

A subset $M\subset[T]$ has the \emph{$\mu$-Baire Property}
if there is an open set $U \subset [T]$, such that differences
$M \setminus U$ and $U \setminus M$ are both $\mu$-meagre.
We need to show that the $\mu$-Baire Property
is shared by somewhat more complicated sets
which appear to be the analogue of analytic subsets of a Polish space.
For this we use the fact
that the class of subsets of $[T]$ having the $\mu$-Baire Property
contains all open sets
and is closed under the following modification
of the Souslin Operation $\mathscr{A}$:
Assign to every sequence $s \in\, {^{<\mu}\mu}$
a subset $P_s$ of
$[T]$. Call this family $(P_s)$ a $\mu$-Souslin scheme.
Then the image of this $\mu$-Souslin scheme under our operation
$\mathscr{A}^\mu$ is given by
$$\mathscr{A}^\mu_\sigma(P_s) :=
\bigcup_{f\in {^{\mu}\mu}}\,\,\bigcap_{\nu<\mu} P_{f\upharpoonright\nu}.$$
Here we can assume that the Souslin scheme is regular
(i.e. $s \subset t \Rightarrow P_t \subseteq P_s$)
and continuous
(i.e. $P_t = \bigcap_{s<t} P_s$ for all limit nodes $t$). 

\begin{thm}[Nikodym's Theorem for weight $\mu$]
\label{thm:nikodym}
Let $\mu$ be a regular cardinal that satisfies $\mu^{<\mu}=\mu$
and let $T$ be a normal $\mu$-splitting and $\mu$-closed tree
of height $\alpha<\mu^+$, $\cf(\alpha)=\mu$.
Then the class of subsets of $[T]$ that possess the $\mu$-Baire Property
is closed under the operation $\mathscr{A}^\mu$.
\end{thm}
\begin{proof}
Check that the proof as carried out
for the case $\mu=\omega$ in \cite[Section 29.C]{kechris}
including all references also works under our circumstances.
\end{proof}

When constructing a homogeneous Souslin tree,
it is convenient to have arbitrarily many symmetries
in the initial segments of the tree.
The following generalization of a lemma of Kurepa (cf.~\cite{kurepa})
provides this.
We will also apply it in the proof of the Reduction Lemma.
\begin{thm}[Kurepa Lemma for regular cardinals]
\label{thm:kurepa}
Let $\mu$ be a regular cardinal that satisfies $\mu^{<\mu}=\mu$,
and let $S,\ T$ be normal $\mu$-splitting and $\mu$-closed trees
of height $\alpha<\mu^+$, $\cf(\alpha)=\mu$.
Then $S$ and $T$ are isomorphic.
\end{thm}
\begin{proof}
We show that the classical back-and-forth argument of Kurepa
also works in the general context
of a regular, possibly uncountable cardinal $\mu = \mu^{<\mu}$.

Let $\mu,\ \alpha,\ S$ and $T$ be as stated in the lemma.
Then $|S| = |T| = \mu$.
So we can pick dense sets $X \subset [S]$ and $Y \subset [T]$,
both of cardinality $\mu$.
Enumerate $X$ and $Y$ by $(x_\nu)_{\nu<\mu}$ and $(y_\nu)_{\nu<\mu}$
respectively.
We construct two maps, a bijection $\overline{\varphi}: X \to Y$ and the
closely related tree isomorphism $\varphi: S \to T$,
such that for all $x\in X$ we will have that
$\overline{\varphi}(x) = \varphi"x$.
The even ordinals $<\mu$ will count the ``forth'' steps
while the ``back'' steps will have odd numbers.
For any ordinal $\nu < \mu$ let
$\overline{\varphi}: X_\nu \to Y_\nu$ be the bijection constructed after
stage $\nu$,
i.e., $X_\nu$ and $Y_\nu$ are the subsets of $X$ and $Y$ respectively
that contain the elements which have been considered in the construction
stages $0¸\ldots,\nu$.

We start the construction by assigning $\overline{\varphi}(x_0)=y_0$ and
$\varphi(x_0\uhr\gamma) = y_0 \uhr \gamma$ for all $\gamma<\alpha$.

We only describe the odd successor steps of the construction,
the even successor steps being symmetric.
So let $\nu<\mu$ by odd (and $\nu-1$ even).
Let $\xi$ be minimal such that
$y_\xi \in Y \setminus Y_{\nu-1}$ and set
$$\gamma :=
\sup\{\beta < \alpha \mid\, (\exists y\in Y_{\nu-1})\,
y \uhr \beta = y_\xi \uhr \beta\}.$$
Since $\alpha$ is of cofinality $\mu > |Y_{\nu-1}|$, we have that $\gamma < \alpha$.
In order to choose a $\overline{\varphi}$-pre-image of $y_\xi$ we have to pick some
branch in $X$ going through $s := \varphi^{-1}(y_\xi \uhr \gamma)$.
But we furthermore have to ensure that no successor of $s$ is an element of some
$x \in X_{\nu-1}$ that is already occupied.
But,
as our trees are $\mu$-splitting and we have $|X_{\nu-1}| < \nu < \mu$,
this is not a problem.
So let $\zeta < \mu$ be minimal such that $x_\zeta \uhr \gamma = s$
and for all $x \in X_{\nu-1}$ we have
$x_\zeta \uhr (\gamma + 1) \ne x \uhr (\gamma + 1)$.
We assign $\overline{\varphi}(x_\zeta) = y_\xi$,
and $\varphi(x_\zeta \uhr \beta) = y_\xi \uhr \beta$ for all $\beta<\alpha$
and define $X_\nu = X_{\nu-1} \cup \{x_\zeta\}$ and $Y_\nu = Y_{\nu-1} \cup \{y_\xi\}$.

If $\lambda < \mu$ is a limit ordinal we just collect what has been fixed so
far and set
$X_\lambda = \bigcup_{\nu < \lambda} X_\nu$ and
$Y_\lambda = \bigcup_{\nu < \lambda} Y_\nu$.

Finally, it is easy to check, that this construction does not break down and
yields a bijective map $\overline{\varphi}: X \to Y$ and an associated
tree isomorphism $\varphi: S \to T$.
\end{proof}

By now we have collected enough facts from descriptive set theory
to prove the Reduction Lemma \ref{lm:reduction} and carry out the
constructions in Sections \ref{sec:rigid-non-rigid} and \ref{sec:unique_non-nice}.

The final two lemmata of this section will be used
in Section \ref{sec:no_schroeder} to design a more involved interplay between
the subalgebra structure and the endomorphisms of the $\kappa$-Souslin
algebra.

\begin{prp}
\label{prp:comeagre_images}
If $X\subseteq[T]$ is $\mu$-comeagre
and $\varphi:[T]\to X$ is continuous, onto and open,
then the images of $\mu$-comeagre subsets of $[T]$
under $\varphi$ are $\mu$-comeagre.
\end{prp}
\begin{proof}
It is clear that every dense subset $D$ of $[T]$ has a dense $\varphi$-image
as the map is onto and continuous. It follows by openness of $\varphi$ that nowhere dense subsets of
$X$ have nowhere dense inverse images. Since the operations of taking unions
and taking pre-images commute, we also have $\mu$-meagre inverse images for
$\mu$-meagre subsets of $X$.
Form this we deduce the claim of the proposition.

So let $M \subset [T]$ be comeagre.
Without loss of generality we can even assume that $M$ is $\mu$-$G_\delta$.
So there is a regular and continuous Souslin scheme $(P_s)$ consisting of
closed sets $P_s \subset [T]$ such that $\mathscr{A}^\mu_s = \varphi" M$.
By the Nikodym's Theorem \ref{thm:nikodym}
we then know that $\varphi" M$ has the $\mu$-BP.
So assume towards a contradiction that there is some $U\subset [T]$ open such
that the intersection of $U$ and $\varphi" M$ is $\mu$-meagre.
For then $\varphi^{-1}"(U\cap\varphi" M) \supset M \cap \varphi^{-1}"U$ is
$\mu$-meagre, contradicting the fact that $M$ is $\mu$-comeagre.
\end{proof}

The proposition just proven in conjunction with the following lemma
will be used to implement an isomorphism between the $\kappa$-Souslin algebra
under construction and one of its $\infty$-nice subalgebras.

\begin{lm}
\label{lm:nice_proj}
Let $\equiv$ be a nice t.e.r. on $T$.
Then the canonical mapping $\pi: T \to T/\!\!\equiv$
induces a continuous map $\overline{\pi}: [T] \to [T/\!\!\equiv]$
and $\overline{\pi}$ is onto and open.
\end{lm}
\begin{proof}
This a straight forward application of the niceness of $\equiv$.
\end{proof}

\subsection{The Reduction Lemma}
\label{sec:reduction}

The Reduction Lemma for t.e.r.s stated below
is a simple observation,
but it will be crucial in $\kappa$-Souslin algebra constructions
that implement nowhere large subalgebras.
It asserts that we can reduce any comeagre subset $M$ of
$[T\uhr\alpha]$ to a comeagre subset
from which we can choose the new level $T_\alpha$ in a way that a given
t.e.r. extends to $T\uhr(\alpha+1)$.
This formulation makes it very flexible, e.g.,
it is no problem to combine the construction of subalgebras with that of
endomorphisms as performed in Section \ref{sec:no_schroeder}.

Recall that given an equivalence relation $\equiv$
on a topological space $\mathscr{X}$
we say that a subset $M\subset \mathscr{X}$ is \emph{suitable for $\equiv$}
if for every  equivalence class $x/\!\!\equiv$ the intersection 
with $M$ is either empty or dense in $x/\!\!\equiv$
(viewed as a subspace of $\mathscr{X}$).
The central idea of the proof will be
to sort out those classes which are not hit by $M$
in a dense subset
and then check that the remaining classes
still form a $\mu$-comeagre set.

\begin{lm}[Reduction Lemma]
\label{lm:reduction}
Assume that $\mu$ is regular such that $\mu^{<\mu}=\mu$ and $\kappa=\mu^+$.
Let $T$ be a $\kappa$-normal and $\mu$-closed tree
of height $\alpha<\kappa$ with $\cf(\alpha)=\mu$
carrying a t.e.r. $\equiv$.
We denote the induced equivalence relation on $[T]$ also by $\equiv$
and assume that for all branches $x\in[T]$ the space
$x/\!\!\equiv$ is $\mu$-Baire.
Furthermore let $M$ be a $\mu$-comeagre subset of $[T]$.
Then the set
$$M' := \{x \in M \mid\,
x/\!\!\equiv\,\,\cap M\text{ is dense in }x/!\!\equiv\,\}$$
is $\mu$-comeagre in $[T]$ and suitable for $\equiv$.
\end{lm}
Note that by Proposition \ref{prp:baire_for_classes} the hypothesis
that the $\equiv$-classes be $\mu$-Baire is no restriction.
\begin{proof}
Without loss of generality we assume that $M$ is a $\mu$-$G_\delta$ set,
i.e. $M=\bigcap_{\nu<\mu} U_\nu$,
where all the $U_\nu$ are dense open in $[T]$.
Define for $\nu<\mu$
$$X_\nu=
\bigcup\left\{x/\!\!\equiv\,\, \mid (x/\!\!\equiv)\cap\,\, U_\nu
\text{ is not dense in }x/\!\!\equiv\right\}.$$
Note that for $x\notin X_\nu$ the set $(x/\!\!\equiv)\cap\,\, U_\nu$
is then open and dense in $x/\!\!\equiv$.
To prove the Reduction Lemma,
we show that $X_\nu$ is $\mu$-meagre for every $\nu<\mu$,
for then  $M'=M\setminus\bigcup_{\nu\in\mu} X_\nu$
is as desired:
For every member $x\in M'$ we then have
$(x/\!\!\equiv)\,\cap\,M'=(x/\!\!\equiv)\,\cap\, M$,
which is by construction of $M'$
a $\mu$-comeagre subset of the $\mu$-Baire set $x/\!\!\equiv$.

In order to show that the sets $X_\nu$ are all $\mu$-meagre
fix $\nu<\mu$ and define for every node $s\in T$
$$Y_s:=\bigcup\left\{x/\!\!\equiv\,\,\mid x \in \hat{s} \text{ and }
(x/\!\!\equiv)\cap\,\, U_\nu\cap\hat{s}=\varnothing\right\}.$$
For every $x\in Y_s$ the basic open set $\hat{s}$ is the witness
of the fact that $(x/\!\!\equiv)\cap\,\, U_\nu$ is not dense in $x/\!\!\equiv$.
Because of $X_\nu=\bigcup_{s\in T}Y_s$ and $|T|=\mu$, it is enough to
show that $Y_s$ is $\mu$-meagre for every $s\in T$.

If we fix $s\in T_\beta$ then of course
$Y_s=\bigcup_{r\in T_\beta}Y_s\cap\hat{r}$.
For all nodes $r\in T_\beta$ with $r\not\equiv s$
we have $Y_s\cap\hat{r}=\emptyset$.
On the other hand the set $Y_s\cap \hat{s}\subset \hat{s}\setminus U_\nu$
is nowhere dense by the definition of $Y_s$.
To prove that the intersections $Y_s\cap \hat{r}$
are $\mu$-meagre also for $r\equiv s$
we claim that $Y_s$ has the $\mu$-Baire Property,
i.e. there is an open set $V\subset[T]$ such that the
differences $V\setminus Y_s$ and $V\setminus Y_s$ are both $\mu$-meagre.
The proof of this claim follows below.
We first apply it to prove the Reduction Lemma.

Along with $Y_s$ the set $Y_s\cap \hat{r}$ has the $\mu$-Baire Property as well,
so either (i) it is $\mu$-meagre or else (ii) there is a node $t>r$,
such that $Y_s\cap\hat{t}$ is $\mu$-comeagre and therefore dense in $\hat{t}$.
Towards a contradiction,
we assume that the second case holds and fix $t$.
Then every node $u$ above $t$
is equivalent to $x\uhr\hgt(u)$
for some $x\in Y_s\cap\hat{s}$.
Our task is
to exhibit a node $u^*$ above $t$
that is equivalent to $y\uhr \hgt(u^*)$
for some $y \in \hat{s}\setminus Y_s$.
This will give the desired contradiction.

Let $u$ be any immediate successor of $t$.
By our assumption there is $w$ above $s$ and equivalent to $u$.
Letting $v:= w\uhr\hgt(t)$ be the immediate predecessor of $w$
we get $v\equiv t$.
Now $Y_s\cap\hat{w}$ is also nowhere dense.
So there certainly is a node $w^*$ above $w$
such that $\widehat{w^*}\cap Y_s$ is empty,
i.e., letting $\gamma=\hgt w^*$ we have
\begin{equation}
(\forall x\in \hat{s}\cap Y_s)\quad w^*\not\equiv x\uhr\gamma.\tag{$\bigstar$}
\end{equation}
Now the honesty of $\equiv$ for the triple $(v,w^*,t)$
(which is not a dispute as $(v,w,t)$ is not
--by the existence of $u\equiv w$ above $t$)
gives us the node $u^* \equiv w^*$ on level $T_\gamma$ above $u$.

Now let $z\in \widehat{u^*}$ be any branch.
We show that $z\not\in Y_s$,
contradicting our assumption
that $Y_s\cap\hat{r}$ is not $\mu$-meagre.
If $z$ was in $Y_s$,
then there would be a branch $x\equiv z$ in $Y_s\cap\hat{s}$.
This in turn would imply that
$$x\uhr \gamma \equiv z \uhr\gamma = u^* \equiv w^*,$$
which is impossible by ($\bigstar$).

Finally we prove that $Y_s$ has the $\mu$-Baire Property.
For this we give a $\mu$-Souslin scheme which consists of open sets
and yields $Y_s$ under the operation $\mathscr{A}^\mu$.
Fix a club $C\subset\alpha$ of order type $\mu$.
To simplify notation we replace ${^{<\mu}\mu}$ by $T\uhr C$ as index set.
(Theorem \ref{thm:kurepa} gives us the necessary tree isomorphism.)

Let $\tilde{U}:=\{ t\in T\mid \, \hat{t} \subseteq U_\nu\}$.
For a node $r\in T\uhr C$ of height $\hgt(r) \leq \hgt(s)$
simply set $P_r= [T]$. For nodes $r$ higher up define
\[P_r:=
\begin{cases}
\bigcup\left\{ \hat{t} \mid\, r\equiv t\right\}, & 
r\uhr\hgt(s)\equiv s \text{ and }
(r/\!\!\equiv)\cap\,\, \tilde{U} = \emptyset ;\\
\emptyset, & \text{otherwise.}
\end{cases}
\]
For every $x\in Y_s$ we easily have $x\in (x/\!\!\equiv) = \bigcap_{\beta\in
  C}P_{x\upharpoonright\beta}$.
If on the other hand
$x\in \bigcap_{\beta\in C} P_{y\upharpoonright\beta}$ for $x,y \in [T]$,
then $x$ and $y$ are equivalent, $x \equiv y$, and thus
$P_{x\upharpoonright\beta} = P_{y\upharpoonright\beta}$ for all $\beta\in C$.
In this case we have
$$x\in \bigcap_{\beta\in C} P_{x\upharpoonright\beta} = (x/\!\!\equiv) \subset Y_s.$$
This finishes the proof.
\end{proof}

\subsection{A rigid Souslin algebra with non-rigid subalgebras}
\label{sec:rigid-non-rigid}

Our first application of the Reduction Lemma is a relatively simple
construction of a rigid $\kappa$-Souslin algebra $\B$
that has a nice and nowhere large subalgebra $\A$.
By Lemma \ref{lm:aut_over_inice} and Theorem \ref{thm:exist_non-nice}
this algebra $\B$ also has non-rigid and nowhere nice subalgebras.
This is opposed to the construction in the following section,
where the explicitely construed subalgebra is nowhere nice
and no non-rigid neither $\infty$-nice subalgebras occur.

\begin{thm}
\label{expl:rig_not_simple}
Assume that $\mu$ is a regular cardinal such that $\mu^{<\mu} = \mu$ 
and $\diamondsuit_{\kappa}(\CF_{\mu})$ hold, where $\kappa := \mu^+$.
Then there is a rigid $\kappa$-Souslin algebra which has an $\infty$-nice
subalgebra.
\end{thm}
\begin{proof}
We aim at constructing a $\kappa$-Souslin tree $T$ with a $\mu$-nice
t.e.r. $\equiv$.
The rigidity of $\B=\RO T$ is obtained by designing $T$ such that
for all club sets $C$ of $\kappa$ the restricted tree $T\uhr C$ is rigid
by a standard argument.
Then by the Restriction Lemma $\B$ will also be rigid.

Let $(R_\nu)_{\nu\in\CF_\mu}$ be a $\diamondsuit_\kappa(\CF_\mu)$-sequence.
We inductively construct $T$ as a $\kappa$-normal,
$\mu$-closed and $\mu$-splitting tree
on the supporting set $\kappa$ along with the t.e.r. $\equiv$.
In successor steps we appoint to each maximal node $\mu$ direct successor nodes
and extend $\equiv$ in any way that maintains the $\mu$-niceness of $\equiv$.

In the limit step $\alpha<\kappa$ we have so far constructed
$T\uhr\alpha$ and $\equiv$ on this tree.
If $\cf(\alpha)<\mu$ we extend every cofinal branch through $T\uhr\alpha$
to level $T_\alpha$.
The new level $T_\alpha$ then has cardinality $\mu^{<\mu}=\mu<\kappa$.
The t.e.r. on level $T_\alpha$ is completely determined by its behavior
on the levels below.

Let now $\alpha$ be of cofinality $\mu$.
We consider the induced equivalence relation $\equiv$ on the space $[T\uhr\alpha]$.
The $\equiv$-classes are perfect (closed and without isolated points) and
non-empty subsets of $[T\uhr\alpha]$.
The level under construction, $T_\alpha$,
corresponds to a dense subset $Q$ of $[T\uhr\alpha]$ of cardinality $\mu$.
In order to obtain a nice extension of $\equiv$ to the new level
we have to choose this subset $Q\subset [T\uhr\alpha]$
such that it is suitable for $\equiv$,
i.e.,  that for every $\equiv$-class $a\subset [T\uhr\alpha]$
the set $a\cap Q$ is either empty or dense in $a$.

Every automorphism $\varphi$ of $T\uhr C$ for some club $C\subset\alpha$
induces an autohomeomor\-phism $\overline{\varphi}$ on $[T\uhr\alpha]$.
In order to achieve a rigid algebra
we have to choose some limit levels
in a way that prevents the potential automorphisms
(proposed by the $\dia$-sequence) from extending to the next level.
This is done by first choosing a branch $x\in [T\uhr\alpha]$
and then the dense set $Q\subset [T\uhr\alpha]$ such that
$x\in Q$ but $\overline{\varphi}(x)\not\in Q$.
(This is a standard argument.)

Now for the choice of $Q$ in the following three cases:
\begin{enumerate}[1)]
\item If $\alpha<\mu\alpha$ or $R_\alpha$ is neither
a maximal antichain of $T\uhr\alpha$ nor
does it code an automorphism of $T\uhr C$ for some club $C$ of $\alpha$,
then we first choose a  dense set $Q_0$ of $[T\uhr\alpha]$ with cardinality $\mu$.
Then let for $x\in Q_0$ be $Q_x$ a dense subset of $x/\!\!\equiv$ of size $\mu$.
Finally set $Q=\bigcup_{x\in Q_0}Q_x$.
\item If $\alpha=\mu\alpha$ and $R_\alpha$ codes an automorphism $\varphi$,
then we and start as in the first case and get $Q'= \bigcup_{x\in Q_0} Q_x$.
Choose $x_0\in Q'$ and set $Q := Q'\setminus\{\overline{\varphi}(x_0)\}$.
Then $Q$ is easily seen to be suitable for $\equiv$ while at the same time
preventing $\varphi$ from extending to $T_\alpha$.
\item If $R_\alpha$ is a maximal antichain of $T\uhr\alpha$,
we want, as in classical Souslin tree constructions under
$\diamondsuit$, that every node of $T_\alpha$ lies above some node of
$R_\alpha$.
The set
$$M := \{ x\in [T\uhr\alpha]\,\mid\, (\exists s\in R_\alpha) s\in x\}$$
of cofinal branches that pass through nodes in
$R_\alpha$ is open dense in $[T\uhr\alpha]$.
We thus can apply the Reduction Lemma \ref{lm:reduction}
and get a $\mu$-comeagre subset $N \subset M$ which is suitable for $\equiv$.
Then we proceed as above, only that all members of $Q$ are chosen from $N$.
\end{enumerate}
Note that we can arrange the coding
such that we do not have to consider a coincidence of cases 2) and 3)
(which would no longer pose a problem anyway).

The result of this recursive construction is a rigid $\kappa$-Souslin tree
carrying the $\mu$-nice t.e.r. $\equiv$ which represents the subalgebra $\A$.
\end{proof}

\section{A lonely nowhere nice subalgebra}
\label{sec:unique_non-nice}

We use the Reduction Lemma to produce a $\kappa$-Souslin algebra $\B$
that essentially has only one subalgebra $\A$
which is furthermore nowhere nice.
In particular $\B$ and all its subalgebras are rigid.
Compare this to the phenomenon of hidden symmetries of
Section \ref{sec:hidden_symm} which occur
whenever there is an $\infty$-nice subalgebra:
While the latter support the paradigm
that subalgebras witness some form of symmetry,
the following construction shows that this is not true
for nowhere nice subalgebras.
This can also be seen in relation to \cite[Theorem 2]{koppelberg-monk}
which exhibits a similar phenomenon in presence of homogeneity.

\begin{thm}\label{thm:unique_non-nice}
Assume that $\mu$ is a regular cardinal such that $\mu^{<\mu} = \mu$ 
and $\diamondsuit_{\kappa}(\CF_{\mu})$ hold, where $\kappa := \mu^+$.
Then there is a $\kappa$-Souslin algebra $\B$
with a nowhere nice subalgebra $\A$ and the following holds:
For every subalgebra $\C$ of $\B$ there is an antichain $F$ of $\B$
such that
$$\C = \B\uhr (-\sum F) \times \prod_{a\in F} (a\A).$$
Moreover, $\B$ is rigid and does not admit $\infty$-nice subalgebras.
\end{thm}
\begin{proof}
This construction resembles the previous one,
the main differences being
\begin{enumerate}
\item the repeated destruction of niceness for $\equiv$
by picking up the key idea
from the proof of Theorem \ref{thm:exist_non-nice} and
\item the more delicate choice procedure on limit levels in order to 
maintain $\equiv$ as an almost nice t.e.r.
while destroying almost all the others.
\end{enumerate}
Let again $(R_\nu)_{\nu\in\CF_\mu}$ denote
our $\diamondsuit_\kappa(\CF_\mu)$-sequence.
We will use it to kill unwanted antichains, t.e.r.s
and tree automorphisms.
The tree $T$ order will again be defined on the supporting set $\kappa$,
and it will be $\mu$-splitting and $\mu$-closed.
The almost nice t.e.r. to represent the nowhere nice subalgebra $\A$ will be
denoted by $\equiv$.

If $\alpha$ is a successor of a successor ordinal
or of a limit ordinal $\gamma$ with $\cf(\gamma)<\mu$
(or if $\alpha=1$),
then we extend $\equiv$ in any way
that maintains almost $\mu$-niceness.
If $\alpha$ is a limit ordinal
whose cofinality is below $\mu$,
then we extend all cofinal branches of $T\uhr\alpha$
to nodes in $T_\alpha$.

We come to the choice of the limit levels
for the case where $\cf(\alpha)=\mu$.
Our diamond set $R_\alpha$ proposes
either a map on $T$ as described below
or it codes a pair $(r,\simeq)$,
where $r\in T\uhr\alpha$ is a node and 
$\simeq$ is a t.e.r. on $(T\uhr C)(r)$ for some $C$ club in $\alpha$
with $\min C = \hgt(r)$.

We first describe how to choose $T_\alpha$
whenever $R_\alpha$ codes
an unwanted symmetry of $T$ or of $T/\!\!\equiv$.
If $R_\alpha$ is an isomorphism
between normal cofinal subtrees of $T\uhr C$ or between
normal cofinal subtrees of $(T/\!\!\equiv)\uhr C$
then we choose the new level
as in Section \ref{sec:rigid-non-rigid} and destroy the isomorphism
while maintaining the honesty of the t.e.r. $\equiv$.
By this precaution we guarantee,
that both $\B$ and $\A$
and also all of their non-trivial relative algebras
won't have any symmetries,
and therefore they will not have any non-trivial large subalgebras.

If $R_\alpha$ proposes a t.e.r. $\simeq$,
which differs from $\equiv$ in an essential manner (see below),
then we have to choose $T_\alpha$ such that $\simeq$
is no longer honest above $r$ when extended to $T_\alpha$.
The node $r$ is introduced to prevent the fatal situation
in which by accident locally
some of the unwanted t.e.r.s survive.
In the end $r$ will vary over all nodes of $T$.

We distinguish between three cases in the manner how
$\simeq$ differs from $\equiv$,
the first being the one of \emph{negligible} difference in which
we (cannot and therefore) will not destroy $\simeq$,
while in the other two cases we can and will prevent $\simeq$ from
extending to the limit level $T_\alpha$.
We list the three cases and describe how $T_\alpha$ is to be chosen.
When mentioning $\simeq$ or $\equiv$ we always mean the equivalence
relation induced on $\hat{r}$.
\noindent
\begin{enumerate}
\item The proposed relation $\simeq$ refines $\equiv$ on $\hat{r}$ in a way
such that there is a maximal antichain $E$ of $T(r)$
such that for all $s\in E$
the  restriction of $\simeq$ to $\hat{s}$ coincides with
either $\equiv$ (restricted to $\hat{s}$) or with $=$, the identity,
and the following holds for all nodes $s,t\in E$,
such that $\equiv$ and $\simeq$ are equal above $s$ and above $t$:
If there are branches $x\in\hat{s}$ and $y\in\hat{t}$
with $x\simeq y$
(and therefore also $x\equiv y$, as $\simeq$ refines $\equiv$)
then the two equivalence relations also coincide on the set
$\hat{s}\cup\hat{t}$.
(Note that this last condition is not a requirement put on $\simeq$ but
on the choice of the antichain $E$; it has to be quite fine.)

We then let
$$M := ( [T\uhr\alpha] \setminus \hat{r} ) \cup \bigcup_{s\in E} \hat{s}$$
and apply the Reduction Lemma \ref{lm:reduction} to get a $\mu$-comeagre
subset $M'$ of $M$ from which we choose $T_\alpha$
such that it is suitable for $\equiv$.
\item Here $\equiv$ on $\hat{r}$
is again refined by $\simeq$ but there is no
antichain as in case (1). Then there must be a node $s$ above $r$
and branch $x$ through $s$
for which $\hat{s} \cap x/\!\!\simeq$
is a nowhere dense subset of $\hat{s} \cap x/\!\!\equiv$.

In this case choose a node $t$ above $s$ such that
$(x/\!\!\simeq) \cap \hat{t} \neq \emptyset$
yet $x/\!\!\simeq\, \not\subset \hat{t}$.
Let
$$M :=  [T\uhr\alpha] \setminus ( \hat{t} \cap x/\!\!\simeq )$$
and note that $M$ is already suitable for $\equiv$.
Choose $T_\alpha$ suitable for $\equiv$
such that at least
one branch in $(x/\!\!\simeq) \,\setminus \hat{t}$ is extended.
Then $\simeq$ is no longer honest as witnessed by the dispute
$(x\uhr\hgt(t),x,t)$.

\item In the last case, $\equiv$ is not refined by $\simeq$.
This means that there is a branch $x$ through $r$ such that
$(x/\!\!\simeq)\,\cap \hat{r}$ is not contained in
$(x/\!\!\equiv)\,\cap \hat{r}$.
Since both sets are closed in $\hat{r}$, their intersection
cannot be a dense subset of $(x/\!\!\simeq) \, \cap \hat{r}$.
This means that we can find a node $s$ above $r$ such that there is a branch
$$  y\in ( \hat{s} \cap x/\!\!\simeq )\,\setminus \,x/\!\!\equiv.$$
Let
$$M := (\, [T\uhr\alpha] \setminus\, x/\!\!\simeq\, ) \cup \,x/\!\!\equiv$$
and apply the Reduction Lemma to get $M'$.
As $M$ contains all of $x/\!\!\equiv$,
this class will still  be present in $M'$.
If we now choose (the branches extended to nodes in)
$T_\alpha$ from $M'$ and make sure that $x$ is extended,
then the dispute $(x \uhr \hgt(s), s, x)$ shows by the choice of $s$ that
$\simeq$ is no longer honest.
\end{enumerate}
Note that the first case also seals maximal antichains.

Now let $\alpha=\gamma+1$ where $\cf(\gamma)=\mu$.
By the inductive hypothesis and our convention on successor levels of $T$
we have so far constructed the tree $T\uhr(\alpha+1)$ and the t.e.r.
$\equiv$ on $T\uhr\alpha$.
Regarding the $\equiv$-classes of level $T_\gamma$ as subspaces of the space
$[T\uhr\gamma]$
we divide each class $a$ in two parts
$a = a_{\mathrm{red}}\,\, \dot{\cup}\,\, a_{\mathrm{green}}$ in way such
that both parts are dense subsets of $a$.
This is analogous to the proofs of Lemma \ref{lm:aut_over_inice} and
Theorem \ref{thm:exist_non-nice} and
gives us a coloring of the whole limit level $T_\gamma$
with colors ``red'' and ``green''.
We then extend $\equiv$ to level $T_\alpha$ such that
\begin{itemize}
\item for 
each limit node $s\in T_\gamma$ its set $\NF(s)$ of direct successors
is partitioned by $\equiv$ into $\mu$ sets of size $\mu$ and
\item if $s,t \in T_\gamma$ are equivalent and of the same color, then 
every successor of $s$ has an equivalent successor of $t$ and
\item if $s,t \in T_\gamma$ are not equivalent or of different colors, then
none of their successors are equivalent.
\end{itemize}
This assures that $\equiv$ is almost nice,
yet $\A = \langle \sum s/\!\!\equiv\rangle^\cm$
will be nowhere nice in $\B = \RO T$.

Having completed the construction of $T$ and $\equiv$,
we now prove that every subalgebra of $\B$
is of the form described in the statement of the theorem.

Let $\C$ be any atomless subalgebra of $\B$ and let $\simeq$ be a t.e.r.
on $T\uhr D'$ for some club $D' \subset \kappa$ representing $\C$.
Let $D$ be the set of $\mu$th order limit points of $D'$,
i.e., let $D = D_\mu$ where
$$D_0 := D',\quad
D_{\nu+1}:=\{\text{limit points of }D_\nu\},\quad
D_\lambda := \bigcap_{\nu<\lambda}\text{ for limit }\lambda.$$
Let $S$ be the stationary set of those ordinals $\alpha\in D$
of cofinality $\cf(\alpha) = \mu$,
such that $\simeq$ on $T\uhr\alpha$ is coded
together with the node $r=\mathbf{root}$ by $R_\alpha$.

By our case split in the construction of $T_\alpha$,
$\simeq$ must always have fallen under case (1) above.
Furthermore there are an ordinal $\alpha^* < \kappa$
and a maximal antichain $E$ of $T\uhr\alpha^*$,
such that $E$ is the antichain which is referred to in case (1)
for all $\alpha>\alpha^*$.
(If this was not true, then at some limit stage of $S$
we would have dropped out of case (1),
thereby destroying $\simeq$.)
We finally have to assemble the elements of $E$ as follows.
Let
$$e:=\sum\{s\in E \mid \,\, \simeq\text{ is equality on }T(s)\},$$
define on the set $E':=\{ s\in E\mid se = 0\}$
of nodes disjoint from $e$ the equivalence relation
$$s\sim t:\iff \left( (\forall\, c\in\C)sc\ne 0 \iff tc \ne 0 \right)$$
and let $F$ be the set of sums $\sum s/\!\!\sim$ for all
$s\in E$ disjoint from $e$.
As $\B\uhr e$ has no large subalgebras but $\B\uhr e$ itself,
we already know that $e\C = \B\uhr e$.
By a similar argument we have $f\C = f\A$ for all $f\in F$.
And as there are no further symmetries between relative algebras
of $\A$ and $\B$, these pieces are relative algebras of $\C$,
i.e., the proof is finished.
\end{proof}

\section{The Schr\"oder-Bernstein Property for Souslin algebras}
\label{sec:schroeder}

Say for a class $\mathscr{C}$ of complete Boolean algebras
that it has the \emph{Schr\"oder-Bernstein Property}
if all pairs $A,B\in\mathscr{C}$ are isomorphic whenever they are
regularly embeddable into each other. 
The classical Schr\"oder-Bernstein Theorem then states that the class of
power set algebras has the Schr\"oder-Bernstein Property.

In his well-known list of set theoretic problems
D. H. Fremlin \cite{fremlin_problems} asked among others the question {\bf FW},
whether or not the class of homogeneous, c.c.c. Boolean algebras has the
Schr\"oder-Bernstein Property.

Farah\footnote{another solution of problem {\bf FW} is due to S.~Geschke}
has solved this problem by pointing
out that (assuming ZFC only) the Cohen algebra $C_{\aleph_2}$,
that adjoins $\aleph_2$ generic Cohen reals,
has a subalgebra $B \not\cong C_{\aleph_2}$,
in which $C_{\aleph_2}$ can be regularly embedded
(cf.~\cite[p.93]{farah_embedding_po}).
By \cite[Proposition 4.1]{farah_embedding_po} it is easy to see
that $B$ is weakly homogeneous and by a theorem of Koppelberg and Solovay
(cf.~\cite[Theorem 18.4.1]{stepanek-rubin}) it is therefore homogeneous.

Here we consider the question whether the class
$\mathscr{S}_{\aleph_1}$ of $\aleph_1$-Souslin algebras
(and more generally, the class $\mathscr{S}_{\kappa}$ of $\kappa$-Souslin
algebras for a regular uncountable cardinal $\kappa$)
has the Schr\"oder-Bernstein Property.
Of course, under Souslin's Hypothesis we have
$\mathscr{S}_{\aleph_1}=\varnothing$ and the answer is trivially affirmative.
But what if there are Souslin algebras?

\begin{thm}
\label{thm:schroeder_independent}
The Schr\"oder-Bernstein Theorem for $\aleph_1$-Souslin algebras is
independent theory {\sf ZFC + $\neg$SH}.
\end{thm}
The proof of this theorem stretches
over the remaining two sections of the paper.
In fact we will show that the Schr\"oder-Bernstein Theorem for $\kappa$-Souslin
algebras consistently fails for every successor cardinal $\kappa$.

\subsection{A model of $\neg${\sf SH} where the Schr\"oder-Bernstein Theorem
for $\aleph_1$-Souslin algebras holds}
\label{sec:schroeder_yes}

The model we use for the first part of our independence proof was constructed
by Abraham and published in \cite[Section 4]{abraham-shelah_aronszajn_trees}.
It is a Jensen-style forcing iteration with the modification that
not all $\aleph_1$-Souslin trees are killed, but one is preserved.
Similar models have also been obtained using the $P_{\max}$-forcing method of
Woodin, cf.~\cite[Section 8]{larson} and \cite[Section 4.0]{shelah-zapletal}.

This preserved tree has a certain property which frequently appears in the
literature on Souslin trees under various names:
Jensen (\cite{GKH}) and Todor\v{c}evi\'c \cite{todorcevic_trees_and_orders}
call these Souslin trees \emph{full} trees,
Abraham and Shelah (\cite{abraham-shelah,abraham-shelah_aronszajn_trees})
use \emph{Souslin trees with all derived trees Souslin},
Fuchs and Hamkins (\cite{degrigST}) denote them as
\emph{($\omega$-fold) Souslin off the generic branch},
while Larson (\cite{larson}), Shelah and Zapletal (\cite{shelah-zapletal})
simply say \emph{free} trees.

We follow the last three authors and
call an $\aleph_1$-normal tree $T$ of height $\omega_1$ \emph{free},
if for every finite antichain $A=\{s_0,\ldots,s_n\}$ the product tree
$$\bigotimes_{k\leq n} T(s_k)$$
is an $\aleph_1$-Souslin tree.
Abraham calls a product tree like this \emph{a tree derived from $T$}
(of dimension $n+1$).
In his model, call it $V$, there is a free tree $R$
such that every $\aleph_1$-Souslin tree contains a copy of one $R$'s derived
trees as a subtree.
But this means that every $\aleph_1$-Souslin tree in $V$ is the tree sum
of countably many trees which are all derived from $R$.

Furthermore, it is easily checked that for a free tree $R$
the class $\mathscr{D}_R$ of countable Cartesian products
of the regular open algebras of
all trees derived from $R$, i.e. the class $(\mathscr{S}_{\aleph_1})^V$
of all $\aleph_1$-Souslin algebras in $V$,
has the  Schr\"oder-Bernstein Property.
So the Schr\"oder-Bernstein Theorem for $\mathscr{S}_{\aleph_1}$ holds
in this model and is therefore consistent to the theory
{\sf ZFC + $\neg$SH}.

\subsection{No Schr\"oder-Bernstein Theorem under $\dia$}
\label{sec:no_schroeder}

To prove the complementary consistency statement for our independence result,
we perform a last $\dia_\kappa$-construction of a $\kappa$-Souslin tree.

\begin{thm}
Assume that $\mu$ is a regular cardinal such that $\mu^{<\mu}=\mu$ and
$\dia_{\mu^+}(\CF_\mu)$ hold. Let $\kappa:=\mu^+$.
Then there is a homogeneous $\kappa$-Souslin algebra $\B$ that has a pair of
$\infty$-nice subalgebras $\A$ and $\C$
such that $\C$ is a subalgebra of $\A$ and
isomorphic to $\B$ yet $\A$ and $\B$ are not isomorphic:
$$\B\, \cong \,\,\C \,\leq\, \A \,\leq\, \B\, ,
\quad \text{yet}\quad  \A \, \not\cong \,\B\,.$$
\end{thm}
\begin{proof}
Let $\mu$ and $\kappa$ be as in the statement of the theorem.
We will construct the $\mu$-closed and $\mu$-normal
Souslinization $T\subset {^{<\kappa}\mu}$ of $\B$ along with
\begin{itemize}
\item $\mu$-nice t.e.r.s $\equiv$ and $\sim$ representing $\C$ and $\A$
  respectively such that $\sim$ refines $\equiv$ in an $\mu$-nice fashion, and
\item a family $(\varphi_{st})$ of tree automorphisms $\varphi_{st}:T\to
  T$ satisfying $\varphi_{st}(s)=t$ for all pairs $s,t$ of nodes of the same
  height in $T$, and
\item a tree isomorphism $\varphi:T/\!\!\equiv\,\to T$; we will define
  $\varphi$ as a map $T\to T$ which is invariant under $\equiv$.

\end{itemize}
Furthermore we diagonalize every potential isomorphism between $\A$ and $\B$,
i.e., between trees $T\uhr C$ and $(T/\!\!\sim)\uhr C$
for every club $C\subset\kappa$.

We start our construction with the trivial level $T_0 = \{\varnothing\}$
carrying only trivial relations.
For the definition of a successor level $T_{\alpha+1}$ out of level $T_\alpha$ we
attach to every node $ s \in T_\alpha \subseteq {^{\alpha}\mu} $ all possible
successors $s^\smallfrown(\nu) = s \cup \{(\alpha,\nu)\} \in {^{\alpha+1}\mu}$,
where $\nu$ ranges over all ordinals less than $\mu$.

In order to extend the t.e.r.s to the new level fix
a partition $ \mathscr{P} = ( P_{\xi,\eta} ) $ of $ \mu $ into $ \mu $ sets
of size $ \mu $ indexed by pairs $ \xi,\eta $ of ordinals less than $ \mu $.
Set $ P_\xi := \bigcup_{\eta<\mu} P_{\xi,\eta}$
and define the extensions of $\equiv$
and $ \sim $ given on $ T \uhr (\alpha+1) $ to $ T_{\alpha+1} $ by letting for
$r,t\in T_\alpha$
$$ r^\smallfrown(\nu) \equiv t^\smallfrown(\lambda) \quad
:\iff r \equiv t \text{ and }
(\exists\, \xi < \mu) \nu,\ \lambda\ \in P_\xi $$
and 
$$ r^\smallfrown(\nu) \sim t^\smallfrown(\lambda) \quad :\iff r \sim t \text{ and }
(\exists\ \xi,\eta < \mu) \nu,\ \lambda \in P_{\xi,\eta}. $$

To extend $\varphi$ to the next level let
$$\varphi( r^\smallfrown(\nu)) = \varphi( r )^\smallfrown(\xi)
\quad\text{ if and only if }\quad \nu \in P_\xi.$$
The tree automorphisms
$\varphi_{st}:T\uhr ( \alpha + 1 ) \to T \uhr ( \alpha + 1 )$
(for $ \hgt(s) = \hgt(t) \leq \alpha $)
are extended to the next level by the simple rule
$\varphi_{st}(r^\smallfrown(\nu))=\varphi_{st}(r)^\smallfrown(\nu)$.
To install for $s=r^\smallfrown(\nu),v=t^\smallfrown(\lambda)\in T_{\alpha+1}$
a new tree automorphism of $T$ we extend
the initial segment $\varphi_{rt} \uhr (T \uhr \alpha+1)$ to $T_{\alpha+1}$
in another direction:
For $w\in T_\alpha$ and $\varepsilon<\mu$ set:
$$\varphi_{sv}(w^\smallfrown(\varepsilon)) = 
\begin{cases}
\varphi_{rt}(w)^\smallfrown(\varepsilon), &
\text{ if }\varepsilon \neq\lambda,\nu\\
\varphi_{rt}(w)^\smallfrown(\mu), & \text{ if }\varepsilon = \lambda\\
\varphi_{rt}(w)^\smallfrown(\lambda), & \text{ if }\varepsilon = \nu.\\
\end{cases}$$
We have finished the successor stage of the construction
of the tree $T$ and the additional structure.
Note that for every $\gamma$
the final segments of the images $\varphi(x)$ and $\varphi_{st}(x)$
of $x\in[T\uhr\alpha]$ beyond $\gamma$
only depend on the final segment of $x$ beyond $\gamma$.

From now on we consider the limit stage $\alpha$.
Once the set of limit nodes on level $\alpha$ is chosen
all mappings and t.e.r.s extend to the new level in a unique way.

We have to find a set $ Q \subset [T\uhr\alpha] $ such that $Q$ is
\begin{enumerate}[(i)]
\item of cardinality $\mu$,
\item dense in $ [ T\uhr\alpha ] $,
\item  closed under the application of the
(homeomorphism of $[T\uhr\alpha]$ induced by the)
tree automorphism  $\overline{\varphi_{st}}$
(and its inverse mapping $\overline{\varphi}_{ts}$)
for all nodes $s,t\in T\uhr\alpha$ of the same height,
\item suitable for $\equiv$ and $\sim$,
\item and closed under the application of the continuous map
  $\overline{\varphi}$ and its inverse mapping in
  the sense that for every $x \in Q$ there are  $u,y \in Q $ such that we have
  $ \overline{\varphi}( u /\!\!\equiv ) = x $
  and $ \overline{\varphi} ( x /\!\!\equiv ) = y $.
\end{enumerate}
We take the members $R_\nu$ of our $\dia_\kappa$-sequence to be subsets or
binary relations
on the initial segments $T\uhr\alpha$ of our tree by virtue of some pre-fixed
bijection between $ \kappa$ and ${^{<\kappa}\mu} \cup {^{<\kappa}(\mu\times\mu)}$.
Let $H$ be the monoid of maps acting on $[T\uhr\alpha]$ generated by the maps
$\overline{\varphi}$ and $\overline{\varphi_{st}}$ for $s,t\in T\uhr\alpha$ of
the same height.
For a point $x\in [T\uhr\alpha]$ let $\orb(x):=\orb_H(x):=\{h(x)\mid\ h\in H\}$
be its orbit under the action of $H$.

If $\alpha < \kappa$ is a limit ordinal such that either $\cf(\alpha)<\mu$ or
the set $R_\alpha$ neither is a maximal antichain of $T \uhr \alpha$ nor does it
induce a tree isomorphism between $(T\uhr C)/\!\!\sim$ and $T \uhr C$ for some
club set $C$ of $\alpha$, then we find a subset $Q \subset [T \uhr \alpha]$
satisying points (i-v) above as follows. 
Choose any branch $y \in [T\uhr\alpha]$ and let $Q_0 := \orb(y)$.
Note that $Q_0$ is already closed under the action of $H$,
dense in $[T\uhr\alpha]$ and suitable for $\sim$ and $\equiv$.
(This follows from the construction as $\orb(x)\cap (x/\!\!\equiv)$ is
dense in $(x/\!\!\equiv)$ and similarly for $\sim$.)
To provide inverse images under the maps from $H$ it suffices
to care about $\overline{\varphi}$
as all the other generators have their inverse in $H$.
For every $x \in [T\uhr\alpha]$ fix a branch $z_x$ such that $\varphi(z_x)=x$.
For $n\in \omega$ let $Q_{n+1}:= Q_n \cup \bigcup_{x\in Q_n} \orb(z_x)$ and
finally $Q=\bigcup Q_n$.

If $R_\alpha$ guesses a maximal antichain $A$ of $ T \uhr \alpha $ and a tree
isomophism $\psi$ between  $(T \uhr C)/\!\!\sim$ and $ T \uhr C$ for some club
$C\subseteq\alpha$
(again we take $\psi$ to be a $\sim$-invariant map $T\uhr C \to T\uhr C$),
then $Q$ also has to satisfy:
\begin{itemize}
\item[(vi)] every branch $x \in Q$ passes through a node $s$ in $A$ and
\item[(vii)] there is a branch $y \in Q$ such that $\overline{\psi}(x) \ne y$
  for all $ x \in Q $.
\end{itemize}
(Of course, in case that $R_\alpha$ only guesses one out of antichain and tree
isomorphism the other corresponding condition is void.)
We satisfy point (vi) by simply restricting our set of potential nodes to
$$K= \{ x\in [T\uhr\alpha]\mid\ (\exists\ s\in A) s\in x\}$$
which is an open dense subset of $[T\uhr\alpha]$.
To fulfill the last requirement is a more subtle task.
The set $Q$ has to be designed around a special branch $y$ which
on one hand must not be reached by $\overline{\psi}$ from within $Q$
and on the other hand brings its orbit
and also some necessary inverse images under $\overline{\varphi}$.
For every branch $y\in K$ we define its \emph{forbidden set}
$$F_y := \bigcup_{h\in H} (\overline{\psi}\circ h)^{-1}(y).$$
First we have to check that there are enough branches
that lie outside of their forbidden sets.
We claim that the set
\begin{eqnarray*}
 L  & := & \{ y \in K \mid\ y \notin F_y \}\\
    &  = & \bigcap_{h\in H}\{ y \in K \mid\ \overline{\psi}(h(y)) \ne y \}
\end{eqnarray*}
  is $\mu$-comeagre in $[T\uhr\alpha]$.
For $h\in H$ the set $\{ y \in K \mid\ \overline{\psi}(h(y)) \ne y \}$
is clearly open in $K$ and, as $\overline{\psi}$ is nowhere 1-to-1, also dense.

Our next step provides $h$-pre-images for all $h\in H$.
Again, it suffices to care about $h=\varphi$.
For a $\mu$-comeagre subset of $[T\uhr\alpha]$, such as our set $L$ defined
above, the subset of branches $x$ such that the whole orbit of $x$ is in $L$,
i.e., $\orb(x)\subset L$, is again $\mu$-comeagre.
Also the set of $\overline{\varphi}$-images of the latter branches is
$\mu$-comeagre by by Proposition \ref{prp:comeagre_images} and
Lemma \ref{lm:nice_proj};
and therefore also the set $M := \bigcap M_n$,
where $M_0 := L$ and
$$M_{n+1} := \overline{\varphi}"\{x \in M_n \mid\ \orb(x)\subset M_n\}.$$
$M$ and all the sets $M_n$ are clearly closed under the action of $H$,
so they  contain $\orb(x)$ along with $x$ and are,
as a consequence, suitable for $\equiv$.
On the other hand, if we are given $ y \in M $,
then for every $n$ the set
$\overline{\varphi}^{-1}(y) \cap M_n$ is $\mu$-comeagre in the space
$\overline{\varphi}^{-1}(y)$
(which is homeomorphic to $[T\uhr\alpha]$ by Kurepas Lemma).
So the set $\overline{\varphi}^{-1}(y) \cap M$ is also $\mu$-comeagre in 
$\overline{\varphi}^{-1}(y)$ and thus non-empty.

Finally we choose any member $y\in M$ and let $N:= M \setminus F_y$.
If we can show that $h"N = N$ for all $h\in H$,
then $N$ is suitable for $\sim$ and $\equiv$
and we can proceed as in the default case
above with $Q_0 := \orb(y)$ and the only further restriction
that the inverse images $z_x$ of $x$ under $\overline{\varphi}$
are to be chosen from $N$.

For every $x\in M$ and $h\in H$ we clearly have that $h(x)\in F_y$ implies
$x \in F_y$, so $h" N\subseteq N$ holds for all $h\in H$.

Concerning the converse inclusion,
it is again enough to consider $h=\overline{\varphi}$,
because $\overline{\varphi}_{st}$ is invertible in $H$ and
therefore done by the first inclusion.
We show that $\overline{\varphi}(x)=z$ and $x\in F_y$ imply
that either $z \in F_y$ or there is a branch $x' \notin F_y$ with
$x'\equiv x$ and thereby $\overline{\varphi}(x')=z$ as well.
So let $\overline{\varphi}(x)=z$ and $\overline{\psi}(h(x))=y$ for some
$h\in H$.
There are three possible types for $h$.
In the first case let $h$ be a homeomorphism, i.e., a concatenation of maps
of type $\overline{\varphi}_{st}$.
Then $h$ locally maps $\equiv$ to $\equiv$ and $\sim$ to $\sim$.
As every $\sim$-class is nowhere dense in its corresponding $\equiv$-class,
$F_y \cap \overline{\varphi}^{-1}(z)$ is $\mu$-comeagre in
$\overline{\varphi}^{-1}(z)$.
So there must be some $x'\in N \cap \overline{\varphi}^{-1}(y)\setminus F_y$.
If in the second case
$h = g \circ \overline{\varphi}$ then we are already done,
for then we have $\overline{\varphi}(g(z)) = y$,
and $z$ lies in the forbidden set.
Finally let $h = h'\circ \overline{\varphi} \circ g$
where $g$ is a homeomorphism.
Here we exploit the fact that $g$ leaves a final segment of $x$ unchanged
and that the components of $\overline{\varphi}(x)$ only depend on the
corresponding components of $x$.
So let $s,t,u\in T_\alpha$ and $v,w$ be such that
$ x = s^\smallfrown v $
and
$ g(x) = t^\smallfrown v $
while
$z = \overline{\varphi}(x) = r^\smallfrown w$
and
$\overline{\varphi}(g(x)) = u^\smallfrown w$.
Then easily
$$\overline{\varphi}(g(x)) = u^\smallfrown w =
\overline{\varphi}_{ru}(r^\smallfrown w) = \overline{\varphi}_{ru}(z).$$
But then $\overline{\psi}\circ h'\circ \overline{\varphi}_{ru} (z) = y$,
so $z$ again is a member of the forbidden set.

This finishes the construction and it can easily be seen
that with $\B := \RO T$, $\A := \langle T/\!\! \sim\rangle$
and $\C := \langle T/\!\! \equiv\rangle$ we have that
$\C$ is $\infty$-nice in $\A$ which in turn is $\infty$-nice in $\B$,
and that $\B$ and $\C$ are isomorphic via $\varphi$ while
$\B$ and $\A$ are not, because all potential isomorphisms between
$T$ and $T/\!\!\sim$ have been diagonalized away.
\end{proof}

\section{Concluding remarks}

Concerning the representation theory of Souslin algebras,
we have not touched here how it can be used to analyze
independent subalgebras and (free) product Souslin algebras.
This can fruitfully be applied, e.g.,
to strongly homogeneous and to free Souslin algebras,
cf. \cite[Sections 1.5-6]{diss}.

Of course, most if not all of the constructions in Part 2 of the present
paper could be carried out without any recourse to topology and
using sophisticated or involved combinatorial arguments.
But, as we hope the constructions performed in the preceeding sections
demonstrate, the topological view substantially simplifies the
diagonalization procedures once the basic notions have been established.
In most cases, it is not hard to see that consistent ``conditions''
imposed on the branches to be extended
are comeagre sets which therefore can freely be combined
(up to $\mu$ conditions at a time).

In \cite{mcsa} (resp. in \cite[Chapter 2]{diss}) we give
a further construction of an $\aleph_1$-Souslin algebra with
aboundant homogeneity properties and with many subalgebras.
We hope that such a construction can in the end be used to construct
a model of ZFC with a unique Souslin line (up to isomorphism).

\section*{Acknowledgments}

I sincerely thank my thesis advisor Sabine Koppelberg,
in particular for her scrutiny in pointing out the gaps in my
na\"ive first arguments.

I would also like to thank Peter Krautzberger
for proof reading large parts of this text
and for his valuable suggestions on how to improve
the presentation.

\bibliographystyle{abbrv}
\bibliography{souslin,buecher}

\end{document}